\documentclass[12pt]{amsart}

\usepackage{latexsym, amssymb, amsmath,ulem,soul,esint}

\usepackage{amsfonts, graphicx}
\usepackage{graphicx,color}
\newcommand{\be}{\begin{equation}}
\newcommand{\ee}{\end{equation}}
\newcommand{\beq}{\begin{eqnarray}}
\newcommand{\eeq}{\end{eqnarray}}

\newtheorem{thm}{Theorem}[section]

\newtheorem{ass}{Assumption}[section]
\newtheorem{lma}{Lemma}[section]
\newtheorem{prop}{Proposition}[section]
\newtheorem{cor}{Corollary}[section]
\newtheorem{defn}{Definition}[section]

\newtheorem{claim}{Claim}[section]
\theoremstyle{remark}
\newtheorem{rem}{Remark}[section]
\numberwithin{equation}{section}

\def\dps{\displaystyle}
\def\div{\mathrm{div}}

\def\be{\begin{equation}}
\def\ee{\end{equation}}
\def\bee{\begin{equation*}}
\def\eee{\end{equation*}}
\def\ol{\overline}
\def\lf{\left}
\def\ri{\right}
\def\bH{\mathbf{H}}
\def\cH{\mathcal{H}}
\def\vn{{\vec\nu}}

\def\by{\mathbf{y}}
\def\bx{\mathbf{x}}
\def\bX{\mathbf{X}}

\def\wn{\wt\nabla}

\def\cI{\mathcal{I}}
\def\bn{\mathbf{n}}
\def\cF{{\mathcal{F}}}

\def\K{K\"ahler }

\def\Ric{\text{\rm Ric}}
\def\Rm{\text{\rm Rm}}

\def\bT{{\mathbf{T}}}

\def\wt{\widetilde}
\def\la{\langle}
\def\ra{\rangle}
\def\p{\partial}

\def\ol{\overline}

\def\e{\epsilon}
\def\a{{\alpha}}
\def\b{{\beta}}

\def\R{\mathbb{R}}

\def\mS{\mathbb{S}}
\def\bF{{\bf F}}

\def\ve{\varepsilon}

\def\n{\nabla}
\def\ppt{\frac{\partial}{\partial t}}
\def\pps{\frac{\partial}{\partial s}}

\def\on{{\overline{\nabla}}}
\def\oR{\,{\overline{\mathrm{R}}}}
\def\oRm{\,{\overline{\mathrm{Rm}}}}
\def\oRic{{\overline{\Ric}}}

\def\un{{\vec\nu}}

\begin{document}

\title[]
{Prescribed mean curvature flow for noncompact hypersurfaces in Lorentz manifolds}

\author{Luen-Fai Tam}
\address[Luen-Fai Tam]{The Institute of Mathematical Sciences and Department of Mathematics, The Chinese University of Hong Kong, Shatin, Hong Kong, China.}
 \email{lftam@math.cuhk.edu.hk}

\renewcommand{\subjclassname}{
  \textup{2010} Mathematics Subject Classification}
\date{\today}

\begin{abstract} Motivated by previous study on mean curvature flow and prescribed mean curvature flow on spatially compact space or asymptotically flat spacetime, in this work we will find sufficient conditions for the short time existence of prescribed mean curvature flow on a Lorentz manifold with a smooth time function starting from  a complete noncompact spacelike hypersurface. Long time existence and convergence will also be discussed. Results will be applied to study some   prescribed mean curvature flows inside the future of the origin in the Minkowski spacetime.  Examples of spacetime related to the existence and convergence results near the future null infinity of the Schwarzschild spacetime are also discussed.

\end{abstract}


\maketitle

\markboth{  Luen-Fai Tam}{Prescribed mean curvature flow in Lorentz manifolds}

\section{Introduction}\label{s-intro}
There are well-developed theories  on mean curvature flow and prescribed mean curvature flow in Euclidean space and in Riemannian manifolds. Classical mean curvature flowed was developed by  of Huisken \cite{Huisken1984} and others.  One may consult   the book  \cite{AndrewsChowGuentherLangford} by Andrews-Chow-Guenther-Langford and the references therein.
On the other hand, people are interested in  constructing maximal or constant mean curvature spacelike hypersurfaces in Lorentz manifolds. Using elliptic method  important results have been obtained by Bartnik \cite{Bartnik1984,Bartnik1988}, Bartnik-Simon \cite{BartnikSimon1982}, Gerhardt \cite{Gerhardt1983}, Treibergs \cite{Treibergs} and Andersson-Iriondo \cite{AnderssonIriondo1999}, to name a few. It is natural to try to use  mean curvature flow or prescribed mean curvature flow to construct such kind of hypersurfaces. In this direction, works have been done by  Ecker-Husiken \cite{EckerHusiken1991},  Gerhardt \cite{Gerhardt2000} in the spatially   compact case, and by Ecker \cite{Ecker1997,Ecker1993} in the noncompact case, for example. In particular, using the constructed maximal surfaces in \cite{Bartnik1984} as barriers,  in \cite{Ecker1993} Ecker was able to use the mean curvature flow to construct spacelike maximal hypersurfaces in asymptotically flat spacetime. Other results on mean curvature flow and prescribed mean curvature flow for noncompact hypersurfaces in Minkowski spacetime have been studied thoroughly in \cite{Ecker1997}.   Recently there are also interesting works by Kr\"ocke et al. \cite{KPLMMMSSV} and Gentile-Vertman \cite{GentileVertman2023} on graphical prescribed mean curvature flows on some product or warped product spacetimes. These motivate us to study prescribed mean curvature flows for complete noncompact spacelike hypersurface in general Lorentz manifolds. In this work, we want to find sufficient conditions for   short time existence, long time existence and convergence of prescribed mean curvature flows in Lorentz manifolds. The results will be obtained by using well developed techniques,    especially those by Bartnik \cite{Bartnik1984},  Ecker-Husiken \cite{EckerHusiken1991} and Ecker \cite{Ecker1993}.

To state our results, let $(N^{n+1},G)$ be a Lorentz manifold with connection $\on$ and curvature tensor $\oRm$ and Ricci tensor $\oRic$. Assume there is a time function $\tau$ on $N$ so that $\on \tau$ is   time like past directed with {\it lapse function} $\a$ given by
 $\a^{-2}=-\la \on\tau,\on\tau\ra$ and let
  $\mathbf{T}=-\a\on\tau$
be the future pointing unit time like vector field.
Suppose $M$ is spacelike hypersurface in $N$ with   future pointing unit normal $\bn$, the {\it tilt factor} of $M$  with respect to  $\mathbf{T}$  is given  by $-\la \bn, \mathbf{T}\ra_G$ and the {\it height function} $u$ of $M$ is the function $u=\tau|_M$. To measure the size of a tensor, as in \cite{Bartnik1984,HawkingEllis}, introduce the  reference Riemannian metric $G_E$ with respect to $\bT$ by
\bee
G_E=G+2\omega\otimes\omega
\eee
where $\omega$ is the one form dual to $\mathbf{T}$. For a tensor $B$ in $N$, the square of its norm at each point with respect to $G_E$ is given  by $\la B,B\ra_{G_E}$ and its norm is denoted by
$|||B|||_{G_E}$. For any  $\Omega\subset N$, we define for $k\ge 0$:

\be\label{e-Omega-norm}
|||B|||_{G_E; k,\Omega}=:\sup_{0\le l\le k}\sup_{\Omega}|||\on^lB|||_{G_E}.
\ee
We will omit the subscript $G_E$ if there is no confusion.

 Let $\bX_0: M\to N$ be an immersed spacelike hypersurface and $\cH$ be a smooth function in $N$, then the prescribed mean curvature flow starting from $\bX_0$ is a map $\bF: M\times[0,s_0)\to N$ satisfying
\be\label{e-mcf-1}
\left\{
  \begin{array}{ll}
    \dps{\pps} \bF(p,s)=[(H-\cH)\vn](p,s)&\hbox{in $M\times[0,s_0)$,}\\
    \bF|_{s=0}=\bX_0,
  \end{array}
\right.
\ee
such that $\bF^s(\cdot)=\bF(\cdot,s)$ is a spacelike immersion, with mean curvature vector $\bH=\vn H  $ where $H$ is the mean curvature,  $\vn$ is the future timelike unit normal. $H, \cH, \vn$ are evaluated at the point $\bF(p,s)$. We also   denote $M_s$ to be the immersed surface given by $\bF^s$. In this work, we will study the questions of short time existence, longtime existence and long time behavior of solutions of \eqref{e-mcf-1} with $M_0$ being noncompact.    For the short time existence, we have the following:

\begin{thm}\label{t-shorttime Lorentz} Let $(N^{n+1},G)$, $n\ge 3$, and $\tau$ be as above. Let $\bX_0: M^n\to N$ be an immersed noncompact spacelike hypersurface with future time like unit normal $\bn$ and second fundamental form $A$ so that  $g_0=\bX_0^*(G)$ is a complete Riemannian metric and let $\cH$ be a smooth function on $N$. Assume the following:

\begin{enumerate}
\item[(i)] $\bX: (-a,a)\times M\to N$ given by $\bX(p,t)=\exp_{\bX_0(p)}(t\bn(\bX_0(p)))$ is defined for some $a>0$;
  \item [(ii)]  $|||\on \bT|||_{0,\Omega} \le c$ for some constant $c$, where   $\Omega=\bX((-a,a)\times M)$;
  \item [(iii)] $||| \oRm|||_{k, \Omega}\le c_k$ for some constant $c_k$   for all $k\ge 0$;
\item[(iv)] $ ||\n^k A||\le a_k$ for some constant $a_k$ on $\bX_0(M)$ for all $k\ge 0$,   where   $\n$ is the covariant derivative of the induced metric $g_0$ and the norm is with respect to this metric;

     \item [(v)] the tilt factor $-\la \bn(\bX_0(p)),\bT\ra\le v_0$  for some $v_0>0$ for all $p\in M$; and
       \item[(vi)] For $k\ge 0$, $||| \cH|||_{k,\Omega}<\infty$.
\end{enumerate}
Then  the prescribed mean curvature flow \eqref{e-mcf-1} has a solution $\bF$ on $M\times[0,\e]$ for some $\e>0$. Moreover,  the immersed hypersurface $M_s$  is complete, the second fundamental forms of $M_s$ together with all covariant derivatives are uniformly bounded in space and time and the tilt factor $-\la \vn(s),\bT\ra$ of $M_s$  is also uniformly bounded in space and time.

\end{thm}

To study  the question of long time existence, let us fix some notation.  Let
$$\tau_-=\inf_N\tau,\ \  \tau_+=\sup_N\tau
$$
which may be infinite.  For any $ \tau_-<  \tau_1<\tau_2<\tau_+$, let
$$\Omega_{\tau_1,\tau_2}=\{p\in N|\ \tau_1<\tau(p)<\tau_2\}.
$$
Consider the following conditions:
\be\label{e-assumption-1}
\left\{
  \begin{array}{ll}
    |||\a|||_{0,\Omega_{\tau_1,\tau_2}}, |||\on \log \a|||_{0,\Omega_{\tau_1,\tau_2}}, |||\on\bT|||_{1,\Omega_{\tau_1,\tau_2}} \hbox{are finite};  \\
    |||\oRm|||_{k,\Omega_{\tau_1,\tau_2}},  \hbox{is finite for all $k\ge 0$};\\
|||\cH|||_{k,\Omega_{\tau_1,\tau_2}}  \hbox{is finite  all $k\ge 0$.}
  \end{array}
\right.
\ee
Since $G_E$ is not complete in general, we  introduce the  condition: There is   a smooth function $\rho>0$ defined on $N$ such that  for any $\tau_-<\tau_1<\tau_2<\tau_+$,
 \be\label{e-rho}
\left\{
  \begin{array}{ll}
     |||\on \rho|||_{\Omega_{\tau_1,\tau_q}}\hbox{ is finite};      \\
   \text {for any $\rho_0>0$ the set $\ol{\Omega}_{\tau_1,\tau_2}\cap\{p\in N|\ \rho(p)\le \rho_0\}$ is compact}.
  \end{array}
\right.
\ee
We have the following characterization of maximum time of existence:
\begin{thm}\label{t-longtime} Let $(N^{n+1},G)$ be a Lorentz manifold with a time function $\tau$ and a smooth function $\rho>0$  and let $\cH$ be a smooth function on $N$ so that they satisfy \eqref{e-assumption-1} and \eqref{e-rho}. In addition, assume $\ol\Ric(w,w)\ge -c^2$ for all unit timelike vector $w$ for    some $c$. Suppose  $\bX_0:M\to N$ is a noncompact immersed spacelike hypersurface with $ \bX_0(M)\subset \Omega_{\tau_1,\tau_2}$   for some $\tau_-<\tau_1<\tau_2<\tau_+$,  such that
 \begin{enumerate}
   \item[(i)] the induced metric on $\bX_0(M)$  is complete so that its second fundamental form $A$ satisfies $\sup_{\bX_0(M)}|\n^kA|<\infty$ for all $k\ge0$;
   \item[(ii)] its tilt factor $\kappa_0$ is uniformly bounded;
   \item[(iii)] $\inf_{M_0}\a >0$, where $\a$ is the lapse function of $\tau$.
 \end{enumerate}
Let  $s_{\max}>0$ be the supremum of $s_0$ so that   the prescribed mean curvature flow \eqref{e-mcf-1}  is defined on $M\times[0,s_0]$ with $\sup_{M\times[0,s_0]}|\n^k A|<\infty$ for all $k\ge 0$, and $\sup_{M\times[0,s_0]}\kappa<\infty$, where $A$ is the second fundamental form of $M_s$ and $\kappa$ is its tilt factor. If $s_{\max}<\infty$ then
$$
\sup_{M\times[0,s_{\max})}\tau(\bF)=\tau_+, \text{\ or\ } \inf_{M\times[0,s_{\max})}\tau(\bF)=\tau_-.
$$
\end{thm}
Finally, we want to study the long time behavior of solution to \eqref{e-mcf-1}. We state two results.

\begin{thm}\label{t-convergence} Let $(N^{n+1},G)$,  $\bX_0$, $\cH$ and $\rho$ be as in Theorem \ref{t-longtime}. Suppose \eqref{e-mcf-1} has a long time solution $\bF$ defined on $M\times[0,\infty)$. In addition, assume the following:
\begin{enumerate}
   \item[(i)] $\tau_-<\tau_1\le \tau(\bF)\le \tau_2<\tau_+$ on $M\times[0,\infty)$ for some $\tau_1, \tau_2$. \item[(ii)] For any $0<s<\infty$ and any $m\ge 0$,
   $$\sup_{M\times [0,s)}(|\n^m A|+\kappa)<\infty,$$
 where $A$ is the second fundamental form of $M_s$ and $\kappa$ is its tilt factor.
\item[(iii)] $\cH\ge a>0$ for some constant $a$,  $\cH$ is monotone in the sense that $\la \on \cH,w\ra\ge0$ for all future directed timelike vector.  $H-\cH\ge - a+\ve_0$ for some $0<\ve_0<a$ on $M_0$.
\item[(iv)] $\oRic(w,w)\ge -c^2$ for any timelike unit vector $w$ with $\frac{\e_0^2}n-c^2>0$.
    \end{enumerate}
Then as $s\to\infty$, $\bF(\cdot,s)$ converges in $C^\infty$ norm in compact sets in $M$ to $\bF_\infty$ so that $\bF_\infty: M\to N$ is an immersed  spacelike hypersurface which is complete with the induced metric and such that the mean curvature $H_\infty$ is equal to $\cH$.
Moreover, if $\a^{-1}$ is uniformly bounded in $\Omega_{\tau_1,\tau_2}$ and if there is a smooth function $\phi\ge1$ on $M_0$ with $\sup_{M_0}(|\n\phi|+|\n^2\phi|)<\infty$ and $\phi\to\infty$ as $\bx\to \infty$ in $M_0$ so that $\sup_{M_0}\phi|H-\cH|<\infty$, then the height functions $u_\infty$ of $M_\infty$ and $u_0$ of $M_0$ satisfy $\lim_{\bx\to\infty}(u_\infty(\bx)-u_0(\bx))=0$.
\end{thm}
We may relax the condition that $\cH\ge a>0$ in the above.

\begin{thm}\label{t-convergence-1} With the same assumptions and notation as in Theorem \ref{t-convergence}, with conditions (iii) and (iv) to be replaced by the following:
\begin{enumerate}
\item[(iii')] $\cH\ge 0$   and $\cH$ is monotone in the sense that $\la \on \cH,w\ra\ge0$ for all future directed timelike vector.  $H-\cH\ge0$ on $M_0$.
\item[(iv')] $\oRic(w,w)\ge 0$ for any timelike unit vector $w$.
    \end{enumerate}
Suppose $\bF(\cdot,s)$ is an embedding for all $s$. Then there exist $s_k\to\infty$ and a spacelike hypersurface $M_\infty$ with mean curvature $\cH$ such that $M_{s_k}\to M_\infty$ in the sense that for any compact $W\subset \ol\Omega_{\tau_1,\tau_2}$ set $d_{h}(M_{s_k}\cap W, M\cap W)\to 0$ as $k\to\infty$, where $d_h$ is the Hausdorff distance defined by $G_E$.
\end{thm}

Let us discuss some of the assumptions in the above theorems. In Theorem \ref{t-shorttime Lorentz}, unlike the case of complete Riemannian manifold, condition (i) in general may not be true, unless, for example, the spacetime is timelike geodesically complete. Condition (ii)  is to ensure that the level sets $t=$constant have uniformly bounded tilt factors, at least for short time under the assumption (v). Together with (iii), (iv) one can construct suitable Banach spaces in order to apply inverse function theorem.  From the proof, one can see that it is sufficient to assume (iii), (iv), (vi) are true up to $k=3$. In case $M_0$ is smooth and compact without boundary, then (iv)--(v) will be satisfied automatically, and the proof can be applied.  We would like to point out that the method of proof of Theorem \ref{t-shorttime Lorentz}, in a more simple way, can be used to obtain short time existence of prescribed mean curvature flow starting from a complete noncompact hypersurface in a complete Riemannian manifold under suitable assumptions.
 In Theorem \ref{t-longtime},  the condition on $\oRic$ is used to obtain estimate of the tilt factor which has been derived by Ecker  \cite{Ecker1993}.    The assumption on the existence of the function $\rho$ is to ensure limit exists.   Condition (iii) is related to condition (i) of Theorem  \ref{t-shorttime Lorentz}, so that one can use that theorem to extend the prescribed mean curvature flow. Roughly speaking,  Theorem \ref{t-longtime} means that longtime existence is true if we have height  estimates, which can be obtained   if there exist suitable barrier surfaces. In Theorem \ref{t-convergence}, conditions (iii) and (iv) are used to assure the flow $\bF$ will converge exponentially fast as in \cite{GentileVertman2023}. (iii) can be replaced by the conditions that $\cH$ is monotone and $\oRic(w,w)\ge c^2>0$ for all unit timelike vector. In Theorem \ref{t-convergence-1}, the result is weaker. We do not claim that $\bF(\cdot,s)$ is convergent as maps. We only consider the convergence of hypersurfaces in terms of Hausdorff distance as in \cite[Theorem 3.8]{Bartnik1988}. Moreover, it is unclear if the limit surface is complete. In case, $\cH\equiv0$, then we do not need to assume $H\ge0$. In this case, this is the result in \cite{Ecker1993}. We should remark by changing the time orientation, dual results are true for Theorems \ref{t-convergence} and \ref{t-convergence-1}. See Remark \ref{r-dual} for more details.

 By the above results, in order to use prescribed mean curvature flow to obtain prescribed mean curvature hypersurfaces, under suitable assumptions, it is sufficient to obtain height  estimates. This can be done if suitable  barrier surfaces exist.   See Corollary \ref{c-existence} and Proposition \ref{t-Mink} for examples. Also, under additional assumptions on the foliation in \cite{AnderssonIriondo1999}, result similar to the result for maximal surface in asymptotically flat spacetime in \cite{Ecker1993} might also true for positive constant mean curvature cut in asymptotically Schwarzschild spacetime which was introduced and studied in   \cite{AnderssonIriondo1999}.

 The organization of this work is as follows: In section \ref{s-L-shortime}  we will prove the short time existence result. We need some  preliminary results  which will be discussed in subsection \ref{ss-setup}. In section \ref{s-long time}, we will prove the long time existence result Theorem \ref{t-longtime} and in section \ref{s-convergence}, we will prove the convergence results Theorems \ref{t-convergence} and \ref{t-convergence-1}.
 In subsection \ref{t-Mink}, we will give some applications on Minkowski spacetime. In subsection \ref{ss-Schwarzschild}, we will   do some computations in Schwarzschild spacetime near the future null infinity, which are related to the conditions in theorems mentioned above  and might be related to the work \cite{AnderssonIriondo1999}. In the appendix, we recall the definitions of certain H\"older spaces, which are used in the proof of short time existence, and we prove a convergence lemma which is used in the proof of Theorem \ref{t-convergence-1}.

\ \

{\it Acknowledgement}: The author would like to thank Albert Chau for useful discussions.

\section{Short time existence}\label{s-L-shortime}
\subsection{Setup and preliminary results}\label{ss-setup}
Before we prove the short time existence Theorem \ref{t-shorttime Lorentz}, we need some preparation. Let $(N,G)$, $\on$, $\oRm$,   $\tau$,  $\a$, $\bT$, $G_E$ be as in section \ref{s-intro}.
First, it is obvious that $G_E$ depends  on $\tau$ and  will be called   the reference   Riemannian metric  with respect to $\tau$.
 If $\tau'$ is another time function with corresponding future pointing unit normal $\mathbf{T}'$, and let $G_E'$ be the reference Riemannian metric with respect to $\tau'$, then $G_E, G_E'$ are different. However, if
$-\la\mathbf{T},\mathbf{T}'\ra$ (which is bounded from below by 1) is uniformly bounded from above then  they will be equivalent by the following result which is well-known and has been used in \cite{Bartnik1984}, see \cite{EckerHusiken1991} for example:

\begin{lma}\label{l-tilt-equivalent}
Let $(V^{n+1},G)$ be a Lorentz vector space. Choose a fixed time orientation. Let $\bT , \bT'$ be two future unit time like unit vectors with dual one forms $\omega, \omega'$. Consider the Euclidean metrics
$G_E=G+2\omega\otimes\omega, G_E'=G+2\omega'\otimes\omega'$. Let $B$ be a tensor   on $V$. Then the norms $|||B|||_{G_E}, |||B|||_{G_E'}$ are equivalent. In fact, if $v=-\la\bT,\bT'\ra$, and $B$ is a $(k,l)$ tensor, then
$$|||B|||_{G_E'}^2\le (n+1)^{k+l}\lf(4v^2\ri)^{k+l}|||B|||_{G_E}^2.$$
\end{lma}
\begin{proof} We sketch the proof. Let $e_0=\bT, e_0'=\bT'$. By choosing an orthonormal basis for the intersection   of $e_0^\perp$ and $(e_0')^\perp$, which are the subspaces perpendicular to $e_0, e_0'$ respectively, we may find $e_1, e_1'$ and $e_2,\cdots, e_n$ so that $e_0, e_1,e_2,\cdots, e_n$ and  $e_0', e_1',e_2,\cdots, e_n$
are orthonormal. It is easy to see that
\bee
e_0=\cosh\lambda\, e_0'+\sinh \lambda \,e_1';\  e_1=\sinh \lambda\, e_0'+\cosh \lambda \,e_1'
\eee
where $v=\cosh\lambda$.
Let $w$ be a vector so that
\bee
w=\sum_{\a=0}^n a_\a e_\a=(a_0\cosh \lambda +a_1\sinh \lambda)\,e_0'+(a_0\sinh \lambda +a_1\cosh \lambda)\,e_1'+\sum_{i=2}^n a_i e_i.
\eee

Hence
\bee
\begin{split}
|||w|||^2_{G_E'}=&(a_0\cosh \lambda +a_1\sinh \lambda)^2+(a_0\sinh \lambda +a_1\cosh \lambda)^2+\sum_{i=2}^n a_i^2\\
\le&2(\cosh^2\lambda+\sinh^2\lambda)\sum_{\a=0}^n a_\a^2\\
\le&4 v^2|||w|||^2_{G_E}.
\end{split}
\eee
 Similarly, $|||w|||^2_{G_E}\le 4 v^2|||w|||^2_{G_E'}$. On the other hand, take a $(0,2)$ tensor $B$ for example. Let $e_\a'$ be orthonormal with $e_0'=\bT'$.
\bee
\begin{split}
|||B|||_{G_E'}^2 =&\sum_{\a,\b}B^2(e_\a',e_\b')\\
\le&  \sum_{\a,\b}|||B|||_{G_E}^2 ||| e_\a'|||_{G_E}^2|||e_\b'|||_{G_E}^2\\
\le&  16v^4 (n+1)^2|||B|||_{G_E}^2.
\end{split}
\eee

The general case can be proved similarly.
\end{proof}

 Let $\bX_0:M\to N$ be an immersion as a spacelike hypersurface. We assume that the induced metric $g_0=\bX_0^*(G)$ on $M$ is complete. Define: $\bX: (-a,a)\times M\to N$ by
\be\label{e-X}
\bX(t,p)=\exp_{\bX_0(p)}(t\bn(\bX_0(p)))
\ee
We assume that there is $a>0$ so that $\bX$ is defined for all $p\in M$ and $t\in (-a,a)$. Let $\Omega=\bX((-a,a)\times M)$. Note that
$$\bX^*(G)=-dt^2+g(t)
$$
where $g(t)$ is the symmetric  tensor  on the slice $M_t=\{t\}\times M$: $g(t)=(\bX^t)^*(\cdot,t)(G)$ where $\bX^t(p)=\bX(t,p)$. In local coordinates $x^1,\cdots,x^n$ of $M$, then $t=x^0, x^1,\cdots, x^n$ will be local coordinates of $(-a,a)\times M$. In the following $\a,\b,\cdots$ run from 0 to $n$ and $i, j, \cdots$ run from $1$ to $n$. We want to find prescribed mean curvature flow which can be expressed as a function $t=w(p,s)$. Namely, the flow up to diffeomorphisms on $M$ is given by
$\bF(p,s)=\exp_{\bX_0(p)}(w(p,s)\bn(\bX_0(p)))$. We   try to solve for $w$. We need some preparations.

\begin{lma}\label{l-tilt} Assume $ |||\on \bT|||_{G_E}\le c$ in $\Omega$.
Let $v=-\la \frac{\p}{\p t},\bT\ra_G$, where we identify $\ppt$ in $(-a,a)\times M$ and $\bX_*(\ppt)$. Suppose $v(p,0)\le v_0$ for all $p\in M$. Then
\bee
v(p,t)\le 2v_0
\eee
for all for $p\in M$ and $|t|\le\dps{\frac1{4cv_0}} $.
\end{lma}
\begin{proof} In the following, we will denote $|||\cdot|||_{G_E}$ with respect to $\bT$ simply  by $|||\cdot|||$.
  $\ppt$ is just the tangent vector of the geodesic $t\to \exp_{\bX_0(p)}(t\bn(\bX_0(p)))$ which is a future pointing unit vector.

\bee
|||\on_{\ppt} \bT|||\le ||| \ppt|||\,|||\on\bT|||\le c|||\ppt|||.
\eee
Since $\la\ppt,\ppt\ra=-1$, by Lemma \ref{l-tilt-equivalent}, we have
\bee
|||\ppt|||^2 \le 2v^2.
\eee
Hence
\bee
\lf|\ppt v\ri|=\lf|\la \on_{\ppt}\bT,\ppt\ra\ri|\le |||\on_{\ppt} \bT|||\, |||\ppt|||\le 2cv^2.
\eee
So if $|t|\le \frac1{4cv_0}$, then $v(t)\le 2v_0$.

\end{proof}

From now on, in the above setting, we always assume the following:

\begin{enumerate}
  \item[(i)] $|||\on\bT|||\le c$;
  \item[(ii)]  $-\la\ppt,\bT\ra\le v_0$ at $t=0$;
  \item[(iii)] $a<\frac1{4cv_0}$ so that $-\la \ppt, \bT\ra\le 2v_0$ for $|t|<a$.
\end{enumerate}

\begin{lma}\label{l-foliation-1}
Assume $c_0=\sup_\Omega|||\oRm|||_{G_E}<\infty$ and   $a_0=\sup_{M_0} |A| <\infty$ where $M_0=\{0\}\times M$  and $A$ is the second fundamental form of the immersed surface given by  $\bX_0$, which will also be denoted by $M_0$.  Then $\bX^t$ is an immersion for $|t|<\min\{a,\frac12 C(n)(a_0^2+c_0 v_0^2)^{-\frac12}\}$,
for some constant $C(n)$ depending only on $n$. Moreover,
$$
e^{-2b_0 t}g_0\le g\le e^{2b_0 t} g_0;\  |A(t)|\le b_0.
$$
where $b_0= 2(a_0^2+c_0v_0^2)^\frac12$.
Here $A(t)$ is the second fundamental form of the hypersurface given by the immersion $\bX^t$.
\end{lma}
\begin{proof}    Let $p\in M$ be fixed.
Let $x^1,\cdots, x^n$ be local coordinates of $p$ and let $t=x^0$. Denote $\bX_*(\frac{\p}{\p x^i})$ simply by $\p_i$ and $\bX_*(\ppt)$ by $\p_t$ or $\p_0$. Since $\p_t$ is the unit tangent vector of the time like geodesic $t\to \exp_{\bX_0(q)}(t\bn(\bX_0(q)))$,
\bee
\ppt\la\p_t,\p_i\ra=\la \on_{\ppt}\ppt,\p_i\ra+\la\ppt,\on_\ppt \p_i\ra=0.
\eee
because $\on_{\ppt}\ppt=0$, $\on_\ppt \p_i=\on_{\p_i}\ppt$ and $\la\ppt,\ppt\ra=-1$.
Let
$
g_{ij}=\la \p_i,\p_j\ra.
$
Then
\be\label{e-ppt g}
\ppt g_{ij}=\la \on_{\p_t}\p_i,\p_j\ra+\la \p_i,\on_{\p_t}\p_j\ra=2A_{ij}
\ee
where $A_{ij}=\la \on_{\p_t}\p_i,\p_j\ra$ which is equal to $\la \p_i, \on_{\p_t}\p_j\ra$ by the above. Suppose $a\ge a_1>0$ is the largest value so that $\bX^t$ is an immersion near $p$ for all $|t|<a_1$. Then $A_{ij}$ is the second fundamental form of the immersed hypersurface defined by $\bX^t$. For such $t$
\bee
\begin{split}
\ppt A_{ij}=&\la\on_{\p_t} \on_{\p_t}\p_i,\p_j\ra+\la\on_{\p_t}  \p_i ,\on_{\p_t},\p_j\ra\\
=&\la\on_{\p_t}\on_{  \p_i}\p_t ,\p_j\ra+\la\on_{\p_t}  \p_i ,\on_{\p_t},\p_j\ra\\
=&\oR_{0i0j}+g^{kl}A_{ik}A_{jl}.
\end{split}
\eee
because $\on_{\p_t}\p_t=0$.
Here $\oR_{0i0j}=\oR(\p_0,\p_i,\p_0,\p_j)$. Then by Lemma \ref{l-tilt-equivalent},
\bee
\begin{split}
\p_t|A|^2=&-2g^{ip}g^{lq}g^{kj}A_{pq}A_{ij}A_{kl}+2g^{il}g^{kj}\oR_{0i0l}A_{kj}\\
\le &2|A|^3+C(n) c_0v_0^2|A|\\
\le &C(n)(|A|^2 +c_0 v_0^2)^\frac32.
\end{split}
\eee
 Hence
\bee
( |A|^2+c_1 v_0^2)^{-\frac12}(0)-( |A| ^2+c_0 v_0^2)^{-\frac12}(t)\le C(n)t.
\eee
and
\bee
( |A |^2+c_1 v_0^2)^{-\frac12}(t)\ge (k^2+c_0 v_0^2)^{-\frac12}-C(n)t.
\eee
Hence if $t\le  \frac12 C(n)(a_0^2+c_0 v_0^2)^{-\frac12}$, then
$$
 |A| (t)\le 2(a_0^2+c_1v_0^2)^\frac12=:b_0
$$
By \eqref{e-ppt g}, we conclude that
$$
e^{-2b_0 t}g_0\le g\le e^{2b_0 t} g_0.
$$
From this it is easy to see that $a\ge a_1\ge  \frac12 C(n)(a_0^2+c_0 v_0^2)^{-\frac12}$. This completes the proof of the lemma.

\end{proof}

From now on we further assume:

\begin{enumerate}
  \item [(iv)] $a$ also satisfies $a\le \frac12 C(n)(a_0^2+c_0 v_0^2)^{-\frac12} $.
\end{enumerate}

Hence $\bX: (-a,a)\times M \to N$ is an immersion. We also denote the immersed surface given by $\bX^t$ by $M_t$.
\vskip.2cm

{\it In the following, in order to study short time existence, we will identify $N$ as $(-a,a)\times M $ and identify  $G$ is  with $\bX^*(G)$. This will simplify the exposition.}

\vskip .2cm

To study the prescribed mean curvature flow, let us first  compute the mean curvature of a spacelike surface $S$ given by a function $w$ on $M$: Namely, the surface is the level zero set $\{f(t,\bx)=0\}$  where $f(t,\bx)=t-w(\bx)$, $\bx\in M$. Here the ambient manifold is $(-a,a)\times M$ with metric $G=-dt^2+g( t)$.

\begin{lma}\label{l-H}
The mean curvature  of $S$ with respect to the future pointing unit normal is given by
\bee
H=(1-|\n w|^2)^{-\frac12}H^o+\div^o
\lf(\frac{\nabla w}{(1-|\n w|^2)^{\frac12}}\ri)
+\ppt  ((1-|\n w|^2)^{-\frac12})
\eee
evaluated at $t=w$.
Here $H^o$ is the mean curvature of the reference slice $M_t=\{t\}\times M$,   $\div^o$ is divergence on this slice, $\n$ is the covariant derivative with respect to $g(t)$ on the slice $M_t$.
\end{lma}

\begin{proof} This follows from \cite[p.160]{Bartnik1984}. We sketch the proof.
 Let $f(t,\bx)=t-w(\bx)$. Then $\on f=\on t-\on w=\on t-\n w$. So the mean curvature of the surface $\{f=0\}$ with respect to the future pointing unit normal is given by
 \bee
 H=\div_G \lf(-\frac{\on f}{(1-|\n w|^2)^\frac12}\ri)\bigg|_{t=w}.
 \eee
 Let $T=\ppt, U=\n w, \lambda=(1-|\n w|^2)^{-\frac12}$. Then $H=\div_G(\lambda(T+U))|_{t=w}$. Now
 \bee
 \begin{split}
 \div_G(\lambda(T+U))=&\div_G(\lambda U)+\div_G(\lambda T)\\
 =&\div^o(\lambda U)+\la \on_{T}(\lambda U),T\ra+\lambda\div_G(T)+T(\lambda)\\
 =&\div^o(\lambda U)+\lambda H^o+T(\lambda)
 \end{split}
 \eee
 because $U\perp T$ and $\on_TT=0$. This completes the proof of the lemma.

\end{proof}

Let $x^i$ be local coordinates of $M$ so that $t,x^i$ are local coordinates of $(-a,a)\times M$. Then

\bee
H^o=  g^{ij}A_{ij}^o=\frac12 g^{ij}\p_t g_{ij}
\eee
\bee
T(\lambda)=\lambda^2 \p_t g^{ij} w_iw_j
\eee

\bee
\begin{split}
\div^o\lf(\frac{\n w}{(1-|\n w|^2)^\frac12}\ri)=&\lambda\lf(  g^{ij}+\lambda^2w^iw^j\ri)w_{;ij}\\
\end{split}
\eee
where $w_{;ij}$ is the Hessian of $w$ in $M_t$. Hence we have the following remark.
\begin{rem}\label{r-mean curvature} The mean curvature is of the form
\be\label{e-mean curvature-1}
\begin{split}
H=&(1-|\n w|^2)^{-\frac12}\lf(  g^{ij}+\frac{w^iw^j}{ 1-|\n w|^2}\ri)w_{;ij}+\frac{w^iw^j}{ 1-|\n w|^2} \p_t g^{ij} \\
&+\frac12(1-|\n w|^2)^{-\frac12} g^{ij}\p_t g_{ij}
\end{split}
\ee
evaluated at $t=w$.
\end{rem}

From the remark, in order to study prescribed mean curvature flow, we need to study the equation involving $H$, which is a second order differential operator with coefficients of the form $\phi(t,\bx)|_{t=w}$ where $\phi(\bx,t)$ can be expressed in terms of $g(t,\bx), g^{-1}(t,\bx)$ and their derivatives. So we need  to estimate $g, g^{-1}$ and their derivatives as functions of $t,\bx$.

\vskip .2cm

In the above setting, let us make the following assumptions. Recall that $\Omega=(-a,a)\times M$ in this setting.

\begin{ass}\label{a-1} Assume:
\begin{enumerate}
  \item [(i)] $|||\on \bT|||_{\Omega}\le c$;
  \item [(ii)] $-\la \bX_*(\ppt),\bT\ra\le v_0$ at $t=0$;
  \item [(iii)] $|||\on^k \oR|||_{\Omega}\le c_k$ for all $k\ge 0$;
\item[(iv)] $ |\n^k A |\le a_k$ on $M_0=\{0\}\times M$ for all $k\ge 0$,  where the $\n$ is the covariant derivative of the metric $g_0$ induced by $G$ and the norm is with respect to this metric.
\item[(v)] $a\le \min\{\frac12 C(n)(k_0^2+c_0 v_0^2)^{-\frac12},\frac12 C(n)(a_0^2+c_0 v_0^2)^{-\frac12}\}$ which are given in Lemmas \ref{l-tilt} and \ref{l-foliation-1}.
\end{enumerate}
\end{ass}
Under this assumption, we identify $N$ with $(-a,a)\times M$ with metric $G=-dt^2+g(t)$.

\begin{lma}\label{l-bounded geometry}
Under the Assumption \ref{a-1}, there is $r>0$ such that for all $p\in M_0$, $\xi_p=:\exp_p:D(r)\to B_p(r)$ is a surjective local diffeomorphism where $D(r)=\{\bx\in\R^n|\  |\bx|<r\}$ and $B_p(r)$ is the geodesic ball of $M_0$ with respect to the   metric $g_0$. Moreover, $C^{-1} g_e\le \xi_p^*(g_0)\le Cg_e$ for some constant $C>0$ independent of $p$, where $g_e$ is the Euclidean metric on $D(r)$, and in the standard coordinates $x^i$ of $D(r)$, $\p^\b(\xi_p^*(g_0)_{ij})$ are uniformly bounded in $D(r)$ independent of $p$ for all multi-index $\b$ and $i, j$.

\end{lma}
\begin{proof} By \cite{Hamilton1995}, it is sufficient to show that the curvature together with its covariant derivatives of $g_0$ are uniformly bounded. But this follows from (ii), (iii), (iv) in Assumption \ref{a-1}, the Gauss equation and Lemma \ref{l-tilt-equivalent}.

\end{proof}

In general, the H\"older norm of a function in $B_p(r)$ dominates the H\"older norm of its pull back under $\xi_p$. On the other hand, if $\xi_p$ is bijective, then these two norms are equivalent. We will   call $\xi_p:D(r)\to B_p(r)$ a {\it quasi-coordinate neighborhood} and the standard coordinates  $x^i$'s of $\R^n$ will be called {\it quasi-coordinates}.
 Suppose   Assumption \ref{a-1} holds, and let $\xi_p:D(r)\to B_p(r)$, $p\in M$ be quasi-coordinate neighborhoods. By Lemma \ref{l-bounded geometry}, one can define the following H\"older spaces for functions on $M\times[0,S_0]$ for $S_0>0$:
\begin{defn}\label{d-holder space}
Let $f$ be a function on $M\times[0,S_0]$ and $k\ge0$ be an integer. $f$ is said to be in $C^{2k+\sigma,k+\frac\sigma2}(M\times[0,S_0])$ for $0<\sigma<1$ if $\xi_p^*(f) \in C^{2k+\sigma,k+\frac\sigma2}(D(r))$ so that its H\"older  norm $||\xi_p^*(f)||_{2k+\sigma,k+\frac\sigma2}$ in $D(r)\times[0,S_0]$ is uniformly bounded independent of $p$.
\end{defn}
For the definition of H\"older space, see appendix. $C^{2k+\sigma,k+\frac\sigma2}(M\times[0,S])$ is a Banach space. We have the following, see \cite[Appendix]{ChauTam2011}:

\begin{lma}\label{l-equiv-holder} Suppose $\xi_p':D(r')\to B_p(r')$ is another family of quasi local coordinates satisfying the same conditions as in Lemma \ref{l-bounded geometry}. Then the H\"older norms defined by $\{\xi_p|\ p\in M\}$ and the H\"older norms defined by $\{\xi'_p|\ p\in M\}$ are equivalent.
\end{lma}

Since $\xi_p$ is a local diffeomorphism, we may consider the standard coordinates of $D(r)$ as local coordinates of points in $B_p(r)$. In the following lemma, since the computations are local, for simplicity, we still denote $\xi_p^*(g(t))$ by $g(t)$.

\begin{lma}\label{l-foliation-2}
In the setting of Lemma \ref{l-foliation-1}. Under Assumption \ref{a-1},  in each quasi-coordinate neighborhood $\xi_p$, any order of  derivatives of $g_{ij}(t)$ with respect to quasi-coordinates and with respect to $t$ are uniformly bounded independent of $p$ and $|t|<a$. Moreover, $C^{-1}g_e\le \xi_p^*(g(t))\le Cg_e$ for some constant $C>0$ independent of $p, t$. Hence any order of derivatives of $g^{ij}(t)$ with respect to quasi-coordinates and with respect to $t$ are uniformly bounded independent of $p$ and $|t|<a$.
\end{lma}
\begin{proof} As mentioned above, we will not distinguish between $\xi_p^*(g(t))$ and $g(t)$. By Lemmas \ref{l-bounded geometry} and \ref{l-foliation-1}, it is easy to see that
\be\label{e-equivalent}
C^{-1}g_e\le \xi_p^*(g(t))\le Cg_e
\ee
in $D(r)$ for some $C$ independent of $p, t$. Hence $g_{ij}(t), g^{ij}(t)$ are uniformly bounded, where $g_{ij}$ are the components of $g$ with respect to the Euclidean coordinates $x^i$.
 By  Lemma \ref{l-foliation-1} and its proof we have
\be\label{e-pt-g}
\p_t g_{ij}=2A_{ij}
\ee
and $|A(t)|$ is uniformly bounded. Hence  $A_{ij}$ and  $|\p_t g_{ij}| $ are uniformly bounded.

 Since all order of derivatives of $g_{ij}(0)$ are uniformly bounded, by \eqref{e-pt-g}, to prove the lemma it is sufficient to prove that $\p_t^k\p^\b A_{ij}$ are uniformly bounded for all $k\ge 0$ and for all multi-index $\b=(\b_1,\b_2,\cdots\b_n)$ with $\b_i\ge0$. Here
$\p^\b=\p^{\b_1}_1\cdots\p^{\b_n}_n$ and $\p_i=\frac{\p}{\p x^i}$. Let $|\b|=\sum_{i=1}^n\a_i$.

We first prove that $\p^\b A_{ij}$ are uniformly bounded for all $\b$. Let $\ol\Gamma(t)=\Gamma(t)-\Gamma(0)$ where $\Gamma(t)$ are the Christoffel symbols of $g(t)$ with respect to the coordinates $x^i$.
If $B_{ijk...}$ is a smooth  tensor on $(-a,a)\times M$ which are tangential to $\{t\}\times M $ for all $t$, for tangential $\p_u$:
\bee
\begin{split}
(\p_t\nabla_u B)_{ijk...}=&\p_t\p_u B_{ijk...}-\p_t( \Gamma_{ui}^pB_{pjk...})-\p_t( \Gamma_{uj}^p B_{ipk...})-\p_t( \Gamma_{uk}^p B_{ijp...})-\dots\\
=&\nabla_u(\p_t B)_{ijk...}+\p_t\ol\Gamma*B.
\end{split}
\eee
because $\p_t\ol\Gamma=\p_t\Gamma$. The number of terms depends only on $n$ and the degree of $B$. We write this simply as:

\bee
(\p_t\n B)=\n\p_t B+\p_t\ol\Gamma*B.
\eee
So
\be\label{e-pptB}
\begin{split}
 \p_t(\n^kB)=&\n\p_t (\n^{k-1}B)+\p_t\Gamma*\n^{k-1}B\\
 =&\n(\n\p_t(\n^{k-2}B)+\p_t\Gamma*\n^{k-2}B)+\p_t\Gamma*\n^{k-1}B\\
 =&\n^2(\p_t((\n^{k-2}B))+\n\p_t\Gamma*\n^{k-2}B+\p_t\Gamma*\n^{k-1}B\\
 =&\cdots\\
 =&\n^k\p_tB+\sum_{i=0}^{k-1}\n^i\p_t\ol\Gamma*\n^{k-i-1}B.
 \end{split}
\ee
Now
\be\label{e-pptA}
\ppt A_{ij}=\oR_{0i0j}+g^{kl}A_{ik}A_{jl}
\ee
and
\be\label{e-pptGamma}
\begin{split}
\ppt\ol\Gamma_{ij}^k=\ppt\Gamma_{ij}^k=g^{kl}(A_{lj;i}+A_{il;j}-A_{ij;l})
\end{split}
\ee
where $;$ is the covariant derivative with respect to $g(t)$ and $0$ is the index corresponding to $\ppt$. So $\oR_{0i0j}=\oR(\ppt, \frac{\p}{\p x^i},\ppt,\frac{\p}{\p x^j})$.
Let $B_{ij}=\oR_{0i0j}$. By Lemma \ref{l-foliation-1} and the choice of $a$, the tilt factor of $M\times\{t\}$ for $|t|<a$ is bounded by $2v_0$. Let $c_k$ be such that $|||\on^k\oRm|||\le c_k$. By \eqref{e-equivalent} and Lemma \ref{l-tilt-equivalent} we have  $| B|\le C(n,v_0,c_0)$. Suppose $| \n^kA |\le b_k$. Since
\bee
\begin{split}
\n_u\oR_{0i0j}=&\on_u\oR_{0i0j}+A_u^k(\oR_{ki0j}+\oR_{0ikj})
\end{split}
\eee
we have $||\n B||\le C(n,v_0,c_0,c_1, b_0)$. Here the norm is taken with respect to $g(t)$. In general, one can show that
\bee
|\n^kB| \le C(n,v_0,c_0,\cdots,c_k; b_0,\cdots,b_{k-1}).
\eee
By \eqref{e-pptGamma},
$$
\n^k\p_t\ol\Gamma= g^{-1}*\n^{k+1}A.
$$
Hence by \eqref{e-pptB} and \eqref{e-pptA}, we have
\bee
\begin{split}
\p_t\n^kA=&\n^k\p_tA+\sum_{i=0}^{k-1}\n^i\p_t\ol\Gamma*\n^{k-i-1}A\\
=&\n^k(B+g^{-1}*A*A)+\sum_{i=0}^{k-1}\n^i\p_t\ol\Gamma*\n^{k-i-1}A.
\end{split}
\eee
\bee
\begin{split}
 |\p_t\n^kA|\le& C(n,v_0,c_0,c_1,\cdots,c_k;b_0,b_1,\cdots, b_{k-1})(1+||\n^kA||)\\
\end{split}
\eee
Hence
\bee
\p_t|\n^kA|^2\le C(1+||\n^kA||^2)
\eee
for some $C=C(n,v_0,c_0,c_1,\cdots,c_k;b_0,b_1,\cdots, b_{k-1})$. Since  $ |\n^kA| (0)$ is bounded, we conclude that $  |\n^kA| $ is bounded   provided $ |\n^pA |$ is bounded for $0\le p\le k-1$. Since $ |A|$ is
bounded, inductively, we have $|\n^kA|$ is bounded for all $k$. By \eqref{e-pptGamma},   and that $ |\n^kA|$ is bounded for all $k$, we have $\n^k\p_t\ol\Gamma$ is bounded. Using \eqref{e-pptB} and  the fact that $\ol\Gamma=0$ at $t=0$, we can  argue as before to show that $\n^k\ol\Gamma$ is bounded for all $k$. From these we conclude that $\p^\b A_{ij}$ are uniformly bounded.


Next, we want to prove that $\p^k_t\p^\b A_{ij}$ are uniformly bounded. Suppose $\p_t^k\p^\b A_{ij}$ are uniformly bounded for all $ \b$, for all $0\le k\le \ell$. Then $\p_t^{k+1}\p^\b g_{ij}$ are uniformly bounded  by \eqref{e-pt-g} and hence $\p_t^{k+1}\p^\b g^{ij}$ for all $\b$ and for all $0\le k\le \ell$. By \eqref{e-pptA}

\be\label{e-induction}
\p_t^{k+1}\p^\b A_{ij}=\p_t^k\p^\b( \oR_{0i0j}+g^{pq}A_{pi}A_{qj})=\p_t^k\p^\b  \oR_{0i0j}+\p_t^k\p^\a (g^{pq}A_{pi}A_{qj}).
\ee
Let $k=\ell$ in the above. Then   second term is uniformly bounded. To estimate the first term, by \eqref{e-induction}, $\p_t^k\p^\b \oR_{0i0j}$ is uniformly bounded for all $\a$ and $0\le k\le \ell-1$.

let $\p_0=\ppt$ and $\p_i=\frac{\p}{\p x^i}$ for $i=1,\cdots, n$.   Since
\bee
\left\{
  \begin{array}{ll}
    \on_{\p_i}\p_j=\Gamma_{ij}^k\p_k +A_{ij}\ppt\\
\on_{\ppt}\p_i=\on_{\p_i}\ppt=g^{kl}A_{il}\p_k\\
\on_{\ppt}\ppt=0,
  \end{array}
\right.
\eee
\bee
\begin{split}
 \on\oRm=&\p\oRm+\oRm*f(A,g,g^{-1},\p g).
\end{split}
\eee
All terms are evaluated using the basis $\p_0, \p_1,\cdots,\p_n$. Here and in the following $f$ is a polynomial  of its arguments which may change from line to line and $\p$ means $\p_\gamma$, $0\le \gamma\le n$. Differentiate  the above relation, we have
\bee
\begin{split}
\on^2\oRm=&\p\on\oR-\on\oRm*f(A,g,\p g,g^{-1})\\
=&\p^2\oRm+(\p \oRm)*f(A,g,\p q,g^{-1})+\oRm*f(A,\p A, g,\p g,\p^2g,g^{-1}).
\end{split}
\eee
Continue in this way, we have:
\bee
\on^k\oRm=\p^k\oRm+\sum_{i=1}^k\p^{k-i}\oRm *f_i
\eee
where $f_i$ denote polynomials in $\p^\b A$ with $|\b|=s$  for $0\le s\le i-1$, $\p^s g$ for $0\le s\le i$ and $g^{-1}$.
First note that $\on^k\oRm$ are uniformly bounded for all $k$ by the facts  that $|||\on^k \oRm|||_{\Omega}$ are uniformly bounded, the tilt factors of $M_t$ are uniformly bounded in space time, and by Lemma \ref{l-tilt-equivalent} and \eqref{e-equivalent}. Since $\p^k_t\p^\a A, \p^{k+1}\p^\a g, \p^k\p^\a\oRm $ are uniformly bounded for $0\le k\le \ell$ and for all $\a$, $\p^k_t\p^\a\oRm$ are uniformly bounded for $0\le k\le \ell-1$, and $g^{-1}$ is uniformly bounded, we conclude that $\p_t^\ell\p^\b\oRm$ are uniformly bounded for all $\a$. This implies that $\p_t^{\ell}\p^\a A$ are uniformly bounded for all $\b$. By induction, the lemma follows.

\end{proof}

\subsection{Proof of short time existence}\label{ss-L-shortime}

Let $N,G$, $\tau$, $M, \bX_0$ be as in Theorem \ref{t-shorttime Lorentz}. Then by Lemmas \ref{l-foliation-1} and \ref{l-foliation-2}, we may assume $N=(-a,a)\times M$ with metric $G=-dt^2+g(t)$ with $a>0$ small enough.
With this setting,   $\bF: M\times[0,S_0]\to N$, $S_0>0$, is a prescribed mean curvature flow starting from $\bX_0:M\to N$ if it satisfies \eqref{e-mcf-1}
such that $\bF(\cdot,s)$ is a spacelike immersion, with mean curvature vector $\bH=\vn H  $ where $H$ is the mean curvature,  $\vn$ is the future timelike unit normal,  and $\cH$ is a   function on $N$. The following fact is well-known.

\begin{lma}\label{l-reduction-1}
Let $\bF$ be a solution of the prescribed mean curvature flow \eqref{e-mcf-1}. Suppose in the setting of Assumption \ref{a-1} there is a smooth family of diffeomorphisms $\phi(p,s)$, $0\le s\le S_0$  of $M$ with $\phi(\cdot,0)$ being the identity map, and a smooth function $w(p,s)$ on $M\times[0,S_0]$ so that in this setting
\bee
\bF(p,s)=\bX( w(\phi(p,s),s), \phi(p,s))=(w(\phi(p,s),s), \phi(p,s))
\eee
with $|w|<a$, where $a$ is as in Assumption \ref{a-1} and $\bX$ is as in \eqref{e-X}. Then $w$ satisfies:
\be\label{e-w-1}
\frac{\p w}{\p s}=-\la \ppt,\vn\ra^{-1}(H-\cH)
\ee
in $M\times[0,S_0]$ and $w(p,0)=0$.
\end{lma}
\begin{proof} Since $|w|<a$ so that $\bX( w(\phi(p,s),s), \phi(p,s))$ is well-defined. In local coordinates $x^i$ of $M$ in the domain and local coordinates $y^i$ of the image of $\phi$,
\bee
\begin{split}
\p_s\bF=&\frac{\p y^j}{\p s}\frac{\p}{\p y^j}+(\frac{\p w}{\p y^j}\frac{\p y^j}{\p s}+w_s)\ppt
\\
=&\frac{\p y^j}{\p s}\lf(\frac{\p}{\p y^j}+ \frac{\p w}{\p y^j}\ppt\ri)+w_s\ppt.
\end{split}
\eee
Here $w_s=\frac{\p w}{\p s}$ as a function of in the  form $w(p,s)$. For fixed $s$, $\phi(p,s)$ is a diffeomorphism of $M$ and so $\bF_*({\frac{\p}{\p x^i}})$ is tangential for all $i$. Since
\bee
\bF_*({\frac{\p}{\p x^i}})=\frac{\p y^k}{\p x^i}\lf(\frac{\p}{\p y^k}+\frac{\p w}{\p y^k} \ppt\ri),
\eee
and $\by=\phi(\bx,s)$ is a diffeomorphism, we have
\bee
\lf(\p_s\bF\ri)^\perp= (w_s\ppt)^\perp
\eee
where $\perp$ is the projection to the normal of immersion given by $\bF^s=\bF(\cdot,s)$.

So the first term in the above is tangential. Hence
\bee
w_s\la \ppt,\vn\ra=\la\p_s\bF,\vn\ra=-(H-\cH).
\eee
From this the result follows because $-\la \vn,\ppt\ra\ge1$.
\end{proof}

The following is also well-known.

\begin{lma}\label{l-reduction-2} Under Assumption \ref{a-1}, suppose $w$ is a smooth function on $M\times[0,S_0]$ such that $|w|<a$, $w(p,0)=0$, and $w\in  C^{2+\sigma,1+\frac\sigma2}(M\times[0,S_0])$ such that $1-|\n w|^2\ge \lambda^{-2}$ uniformly for some constant $\lambda>0$. Here $\n $ is the gradient on the slice $\{t\}\times M $.  Let $\wt\bF(p,s)=\bX(p, w(p,s))$. Then the $\wt\bF(\cdot,s)$ is an immersion for all $s$. Suppose $w$ satisfies \eqref{e-w-1} with mean curvature vector $\bH=\vn H$ with $\cH$ being bounded.
 Then one can find  a smooth family of diffeomorphisms $\phi(p,s)$, $0\le s\le S_0$  of $M$ with $\phi(\cdot,0)$ being the identity map such that
$$\bF(p,s)=\wt\bF(\phi(p,s), w(\phi(p,s),s))$$
is a smooth solution to \eqref{e-mcf-1}.
\end{lma}
\begin{proof}  Denote $\wt\bF(\cdot,s)$ by $\wt\bF^s(\cdot)$. Let $y^i$ be local coordinates of $M$ so that $y^i, t$ are local coordinates of $N=(-a,a)\times M$. We may assume that $y^i$ are quasi local coordinates. Then $\wt\bF$ is of the form
$$
\wt\bF(\by,s)=(\by, w(\by,s)).
$$
 Let
$e_i=:\wt\bF^s_*(\frac{\p}{\p y^i})=\frac{\p}{\p y^i}+w_i\ppt$ where $w_i=\frac{\p w}{\p y^i}$. So $\la e_i,e_j\ra=g_{ij}(\by, t)|_{t=w}-w_iw_j$. Let $\xi_i, 1\le i\le n$ be real numbers and let $\xi^i=g^{ik}\xi_k$. Then
\bee
\la e_i,e_j\ra \xi^i\xi^j=\xi^i\xi_i-\lf(\sum_{i=1}^n w_i\xi^i\ri)^2.
\eee
So
\bee
\xi^i\xi_i\ge \la e_i,e_j\ra \xi^i\xi^j\ge \xi^i\xi_i(1-|\n w|^2)\ge \lambda^{-2}\xi^i\xi_i.
\eee
Hence the graph $\wt\bF^s$ is spacelike with induced metric uniformly equivalent to $g(\by,t)|_{t=w}$ at $\wt\bF^s(\by)$, and hence is equivalent to $g(\by,0)$ by Lemma \ref{l-foliation-1}. From this one can see that there is a smooth family of vector fields
$V(\by,s)$ on $M$ $s\in [0,S_0]$ such that $\wt\bF^s_*(V)=\lf(\p_s\wt\bF\ri)^\top$ which is the tangential component of $ \p_s\wt\bF $ to the graph of $\wt\bF^s$. To be more precise, let $\vn$ be the future timelike unit normal of the graph of $\wt\bF^s$,
\bee
\p_s\wt\bF=  w_s\ppt= (w_s\ppt)^\top-  w_s \la\vn,\ppt\ra\vn=( w_s\ppt)^\top+(H-\cH)\vn
\eee
because $w$ satisfies \eqref{e-w-1}. Now $( w_s\ppt)^T(\by,s)=v^i(\by,s)e_i$ where $e_i$ are tangent to the graph of $\wt\bF^s$ as before. Then $V=v^i\frac{\p}{\p y^i}$. It is easy to see that $V$ is smooth in $\by, s$. Let $\phi(p,s)$ be the integral curve of $-V$ so that $\phi(\cdot,0)$ is the identity map. We claim that $|V|_{g(0)}$ is uniformly bounded. If the claim is true, then $\phi$ can be defined on $M\times [0,S_0]$. In this case, let
\bee
\bF(p,s)=\wt\bF(q, w(q,s))|_{q=\phi(p,s)}.
\eee
Then $\bF$ is smooth, $\bF(p,0)=\bX_0(p)$ because $w(p,0)=0$. Moreover,
\bee
\begin{split}
\p_s\bF=&\wt\bF^s_*(\p_s\phi(p,s))+(\p_s\wt\bF)^\top+(H-\cH)\un\\
=&(H-\cH)\un.
\end{split}
\eee
The lemma follows provided we can prove the claim that $||V||_{g(0)}$ is bounded. To prove the claim, let $e_i$ be as before,
\bee
\begin{split}
\la(\p_s\wt\bF)^\top,e_i\ra=&w_i w_s
\end{split}
\eee
which is uniformly bounded because $w\in C^{2+\sigma,1+\frac\sigma2}(M\times[0,S_0])$, $H$ is given by \eqref{e-mean curvature-1}, $1-|\n w|^2\ge\lambda^{-2}$, and $\cH$ is uniformly bounded. Here we have also used Lemma \ref{l-foliation-2}. Since the metric $\la e_i,e_j\ra$ is uniformly equivalent to $g_{ij}(0)$, this implies the claim is true.

\end{proof}

Hence in order to find short time existence of the prescribed mean curvature flow \eqref{e-mcf-1}, it is sufficient to find $w$ satisfying the conditions in the above lemma on $M\times[0,S_0]$ for some $S_0>0$.
We will use the inverse theorem on $C^1$ maps on Banach spaces. We proceed as in \cite{ChauTam2011} (see also \cite{ChauChenHe}) using the argument outlined in \cite{Hamilton1982}. Consider the Banach space
\bee
 C^{2+\sigma,1+\frac\sigma2}_0(M\times[0,S_0])=\{f\in  C^{2+\sigma,1+\frac\sigma2}(M\times[0,S_0])|\ \ f(\cdot,0)=0\}
\eee
for some $0<\sigma<1$ and $S_0$ will be chosen later. If there is no confusion, we will drop $M\times[0,S_0]$ in the definition of H\"older spaces.

Let $\mathcal{B}$ be the open set in $ C_0^{2+\sigma,1+\frac\sigma2}(M\times[0,S_0])$
consisting functions $f$ such that

\be\label{e-Banach}
  \sup_{M\times[0,S_0]} |f|<a \  \mathrm{and}\
   \inf_{M\times[0,S_0]} (1-|\n f|^2)> \frac12
\ee
where $|\n f|^2$ is the norm of the gradient with respect to $g(t)$. More precisely in local coordinates $x^i$ of $M$,
\bee
|\n f|^2(\bx,s) =f^if_i
\eee
where $f_i=\p_i f=\frac{\p}{\p x^i}f$ and $f^i(\bx,s)=g^{ij}(\bx,t)|_{t=f(\bx,s)} f_i(\bx,s)$.

We want to define a suitable map $\mathcal{F}$ from $\mathcal{B}$ to $C^{\sigma,\frac\sigma2}(M\times[0,S_0]$. Let $w\in \mathcal{B}$ consider the graph of $w$ in $N$ which is identified with $(-a,a)\times M$ with metric $  G=-dt^2+g(t)$. The graph of $w(\cdot,s)$ is given by $(\bx, w(\bx,s))\in (-a,a)\times M$, $\bx\in M$. This is well defined because $|w|<a$. The graph is spacelike because $1-|\n w|^2>\frac12$. For fixed $s$, the future pointing unit normal is given by
$$
\vn(w)=\dps{\frac{\ppt+\n w}{(1-|\n w|^2)^{\frac12}}}.
$$
The mean curvature vector is $\bH=H\un$. Let $\cH$ be a smooth function in $(-a,a)\times M$. Define
\be\label{e-F-1}
\mathcal{F}(w)=\frac{\p w}{\p s}+\la \ppt,\vn(w)\ra^{-1}(H-\cH)
=\frac{\p w}{\p s}-(1-|\n w|^2)^\frac12(H-\cH)
\ee
for $w\in \mathcal{B}$. In the following, it is more convenient to use the quasi local coordinates to express the map. By \eqref{e-mean curvature-1}, we have
\be\label{e-F-2}
\begin{split}
\mathcal{F}(w)=&\frac{\p w}{\p s}-  \lf(  g^{ij}+\frac{ w^iw^j}{1-|\n w|^2}\ri)w_{;ij}-\frac{w_iw_j}{(1-|\n w|^2)^\frac12}  \p_t g^{ij} \\
&-\frac12  g^{ij}\p_t g_{ij}+ (1-|\n w|^2)^\frac12 \cH\\
=&\frac{\p w}{\p s}-  \lf(  g^{ij}+\frac{ w^iw^j}{1-|\n w|^2}\ri)\lf(w_{ij}-\Gamma_{ij}^kw_k\ri) -\frac{w_iw_j}{(1-|\n w|^2)^\frac12}  \p_t g^{ij} \\
&-\frac12  g^{ij}\p_t g_{ij}+ (1-|\n w|^2)^\frac12 \cH.
\end{split}
\ee
For any function $f(t,\bx)$  in the above it is understood the function will be evaluated at $t=w$. Here $\Gamma_{ij}^k$ is the Christoffel symbols of $g$ with respect to the quasi local coordinates.

\begin{lma}\label{l-F-range} Suppose $|||\on^k\cH|||_{\Omega}$ is uniformly bounded for $k=0,1$, then $\mathcal{F}(w)\in C^{\sigma,\frac\sigma2}$ for $w\in \mathcal{B}$.
\end{lma}
\begin{proof} Let $\xi_p: D(r)\to B_p(r)$ be quasi local coordinate neighborhood. By Lemma \ref{l-foliation-2}  all the derivatives of the components of  $\xi_p^*(g(t))$ and its inverse with respect to $t$ and the quasi local coordinates $x^i$ are uniformly bounded independent of $p$ and for $|t|<a$. Moreover, the $C^{2+\sigma,1+ \frac\sigma2}$ norm of $\xi^*(w)$ in $D(r)\times[0,S_0]$ is uniformly bounded independent of $p$. In order to prove the lemma, it is sufficient to prove that the $C^{ \sigma,  \frac\sigma2}$ norm in $D(r)\times[0,S_0]$ of each term of the last expression in \eqref{e-F-2} is uniformly bounded, where $w$ is understood to be $\xi^*_p(w)$ and $g$ is understood to be $\xi^*_p(g)$ etc. By the assumption on $w$, the $C^{2\sigma,\frac\sigma2}$ norms of $w_i, w_{ij}$ are uniformly bounded. On the other hand, since the first derivative of $g_{ij}$ in $x^i, t$ are uniformly bounded and since $w_i, w_s$ are uniformly bounded, we conclude that the $C^{\sigma,\frac\sigma2}$ norms of $g_{ij}( w(\bx,s),\bx)$ are uniformly bounded. Similarly,  the $C^{\sigma,\frac\sigma2}$ norms of $g^{ij}, \p_t g_{ij}, \p_tg^{ij}, \Gamma_{ij}^k, g$ are uniformly bounded evaluated at $t=w(\bx,s)$.  By the assumption on $w$ and   Lemma \ref{l-foliation-2}, $w^i=g^{ik}w_k$ has uniformly bounded $C^{ \sigma,\frac\sigma2}$ norm.
Since $1-|\n w|^2>\frac12$, one can see that for any $c$,  the  $C^{2\sigma,\frac\sigma2}$ norm of $(1-|\n w|^2)^c=(1-w^iw_i)^c$ are uniformly bounded. Since $|||\on^k\cH|||_{\Omega}$ is uniformly bounded for $k=0,1$, $\cH(w(\bx,s),\bx)$ is in $C^{\sigma,\frac\sigma2}$ by Lemmas \ref{l-foliation-1} and \ref{l-foliation-2}. From these, we conclude that the lemma is true.
\end{proof}

Next we want to compute the differential of $\cF$. Let $v(\bx,s)$ be a function on $M\times[0,S_0]$ with $|v|<a$. For any family of  tensors $T(t,\bx)$ in $\bx$ for  $(t,\bx)\in (-a,a)\times M$, denote
 \bee
\,^vT(\bx,s)=:T(v(\bx,s),\bx).
\eee
For examples:

\bee
\left\{
  \begin{array}{ll}
    v^i =\,^vg^{ij} v_j \\
v_{;ij} =v_{ij} -\,^v\Gamma_{ij}^k v_k. \\
  \end{array}
\right.
\eee
We will compute $D\cF$ at $v$ in quasi local coordinates $x^i, s$ in $M\times[0,S_0]$.

\begin{lma}\label{l-DF} In addition to Assumption \ref{a-1}, suppose $|||\on^k\cH|||_{\Omega}$ are uniformly bounded for all $k\ge 0$. Then in the quasi local coordinates, for any $v\in \mathcal{B}$, the differential $D\cF_v$ of $\cF$ at $v$ is given by:
\be\label{e-DF}
D\mathcal{F}_v(\phi)=\phi_s -\,^va^{ij}\,\phi_{ij}+\,^vb^k\,\phi_k+\,^vc\,\phi
\ee
where
\be\label{e-abc}
\left\{
  \begin{array}{ll}
    \,^v a^{ij}=&\,^v g^{ij} +(1-|\n v|^2)^{-1}v^iv^j; \\
 \,^v b^k
=&\dps{- \frac{2 (    \,^v g^{ki}v^j (1-|\n v|^2)+v^iv^j  v^k) (v_{ij}-\,^v\Gamma_{ij}^mv_m)}{(1-|\n v|^2)^2}+\lf(  \,^vg^{ij}+\frac{ v^iv^j}{1-|\n v|^2}\ri) \,^v\Gamma_{ij}^k } \\
&2\dps{-\bigg[\frac{  \,^v\p_tg^{ij}v_iv_j v^k+(1-|\n v|^2)\,^v\p_tg^{ik}v_i+ v^k((1-|\n v|^2) \,^v\cH}{(1-|\n v|^2)^{ \frac32}}\bigg]}\\
 \,^v c=&\dps{-\bigg[   \ ^v\p_tg^{ij}+\frac{2  \ ^v \p_tg^{ik}v_k }{1-|\n v|^2}+\frac{v^iv^j  \ ^v \p_tg^{lm}v_lv_m }{(1-|\n v|^2)^2}\bigg](v_{ij}-\,^v\Gamma_{ij}^kv_k)}\\
&\dps{+\lf(  \,^vg^{ij}+\frac{ v^iv^j}{1-|\n v|^2}\ri)  \ ^v\p_t\Gamma_{ij}^kv_k } \dps{-\bigg[\frac{  \ ^v \p_tg^{lm}v_lv_m \,^v\p_tg^{ij}}{(1-|\n v|^2)^{ \frac32}} + \frac{ \,^v \p_t^2 g^{ij}}{(1-|\n v|^2)^\frac12} \bigg]v_iv_j}
\\
&\dps{-\frac12 \lf(\,^v\p_t g^{ij}\,^v\p_tg_{ij}+\,^vg^{ij}\,^v\p_t^2g_{ij}\ri)-\frac{  \ ^v \p_tg^{lm}v_lv_m\,^v\cH}{(1-|\n v|^2)^\frac12}
+ (1-|\n v|^2)^\frac12 \,^v\p_t\cH}
 \\
  \end{array}
\right.
\ee
Moreover, $\cF$ is $C^1$ in $\mathcal{B}$.
\end{lma}
\begin{proof}
Fix $v\in  \mathcal{B}$. Let $\phi\in C_0^{ 2+\a, 1+\frac\a2 }$. Here and below the H\"oder spaces are defined on   $M\times[0,S_0]$. Then for $|\e|$ small, $v+\e\phi\in C_0^{2+\a, 1+\frac\a2} $ and $\mathcal{F}(v+\e\phi)$ is in $C^{\a,\frac\a2} $. In fact, for $|\e|$ small enough, $v_\e=:v+\e\phi\in \mathcal{B}$.
Then at $\e=0$,
\bee
\left\{
  \begin{array}{ll}
\dps{\frac{\p}{\p \e}(v+\e\phi)_s=\phi_s}\\
   \dps{ \frac{\p}{\p \e}\lf( \,^{v_\e}g^{ij}\ri)=\phi  \ ^v\p_tg^{ij}}  \\
\dps{ \frac{\p}{\p \e}\lf( \,^{v_\e}\p_tg^{ij}\ri)=\phi \ ^v \p_t^2g^{ij}} \\
\dps{\frac{\p}{\p \e}v_\e^i= \phi \ ^v \p_tg^{ik}v_k +\phi_k  \ ^v g^{ki} }\\
\dps{\frac{\p}{\p \e}\,^{v_\e}  \mathcal{H} = \phi \ ^v\p_t\mathcal{H}}\\
\dps{\frac{\p}{\p \e}\,^v\Gamma_{ij}^k=\phi\ ^v\p_t\Gamma_{ij}^k.}
  \end{array}
\right.
\eee
In the following $f'$ means $\frac{\p f}{\p\e}|_{\e=0}$. We have:
\bee
\begin{split}
&\lf[\lf(  g^{ij}+\frac{ v^iv^j}{1-|\n v|^2}\ri)\lf(v_{ij}-\Gamma_{ij}^kv_k\ri)\ri]'\\
=&\bigg[\phi  \ ^v\p_tg^{ij}+\frac{2(\phi \ ^v \p_tg^{ik}v_k +\phi_k  \ ^v g^{ki}v^j)}{1-|\n v|^2}+\frac{v^iv^j(\phi \ ^v \p_tg^{lm}v_lv_m+2\phi_k v^k) }{(1-|\n v|^2)^2}\bigg](v_{ij}-\,^v\Gamma_{ij}^kv_k)\\
&+\lf(  \,^vg^{ij}+\frac{ v^iv^j}{1-|\n v|^2}\ri)\lf(\phi_{ij}-\,^v\Gamma_{ij}^k\phi_k-\phi\ ^v\p_t\Gamma_{ij}^kv_k\ri)
\end{split}
\eee
\bee
\begin{split}
\lf(\frac1{(1-|\n v|^2)^\frac12}  \p_t g^{ij} v_iv_j\ri)'=&\frac{\phi \ ^v \p_tg^{lm}v_lv_m+2\phi_k v^k}{(1-|\n v|^2)^{ \frac32}}\,^v\p_tg^{ij}v_iv_j\\
&+ \frac1{(1-|\n v|^2)^\frac12} \lf(\phi\,^v \p_t^2 g^{ij} v_iv_j
+2\,^v\p_tg^{ik}v_i\phi_k\ri)
\end{split}
\eee
\bee
\begin{split}
(\frac12  g^{ij}\p_t g_{ij})'=&\frac12\phi\lf(\,^v\p_t g^{ij}\,^v\p_tg_{ij}+\,^vg^{ij}\,^v\p_t^2g_{ij}\ri)
\end{split}
\eee
\bee
\begin{split}
\lf( (1-|\n v|^2)^\frac12 \cH\ri)'=&-\frac{\phi \ ^v \p_tg^{lm}v_lv_m+2\phi_k v^k}{(1-|\n v|^2)^\frac12}\,^v\cH
+\phi(1-|\n v|^2)^\frac12 \,^v\p_t\cH.
\end{split}
\eee
Let $D\cF_v(\phi)=\frac{\p}{\p\e}\cF(v+\e\phi)|_{\e=0}$. Then
\bee
D\cF_v(\phi)=\phi_s -\,^va^{ij}\phi_{ij}+\,^vb^k\phi_k+\,^vc\phi,
\eee
where
  $\,^va^{ij}, \,^vb^k, \,^vc$ are as in \eqref{e-abc}.
 Since $|||\on^k\cH|||_{\Omega}$ is uniformly bounded for all $k$,  together with Lemma \ref{l-foliation-2}, all the derivatives of $g_{ij}$ and $\cH$ with respect to $t$ and the quasi local coordinates $x^i$ are uniformly bounded. Here we identify $g$, $\cH$ etc to their pull back to $D(r)$. From this
 and the fact that $v\in C^{2+\sigma,1+\frac\sigma2}_0$, $D\cF_v$  is a bounded linear operator from  $C^{2+\sigma,1+\frac\sigma2}_0$ to $C^{\sigma,\frac\sigma2}$. Moreover, by the structure of $\,^va^{ij}, \,^vb^k, \, ^vc$, we have for $u, v\in \mathcal{B}$, the $C^{\sigma,\frac\sigma2}(D(r)\times[0,S_0])$ norms of
   $ \,^ua^{ij}-\,^va^{ij} $, $\,^ub^k- \,^vb^k,$ and $\,^uc- \, ^vc$ are bounded by $C||u-v||_{C^{2+\sigma,1+\frac\sigma2}_0(M\times[0,S_0])}$ for some constant $C$ independent of the coordinate neighborhoods. Here $C$ may depend on the upper bound of the norms of $u, v$ in $C^{2+\sigma,1+\frac\sigma2}_0(M\times[0,S_0])$. Hence the norms of the operators $D\cF_u, D\cF_v$ satisfy:
   $$
   ||D\cF_u- D\cF_v||\le C||u-v||_{C^{2+\sigma,1+\frac\sigma2}_0(M\times[0,S_0])}.
   $$
   This implies $D\cF_v$ is continuous in $v$. Here we have used the fact that $(1-|\n w|^2)>\frac12$ for all $w\in \mathcal{B}.$ We claim that $D\cF_v$ is the differential of $\cF$ at $v$.  Let $v\in \mathcal{B}$ be fixed and $h\in \mathcal{B}$ so that $||h||_{C^{2+\sigma,1+\frac\sigma2}_0(M\times[0,S_0])}$ (which will be denoted simply by $||h||$ in the following) is small enough. Then

\bee
\begin{split}
\cF(v+h)-\cF(v)-D\cF_v(h)=&\int_0^1\frac{d}{d \e}\cF(v+\e h) d\e -D\cF_v(h)\\
=&\int_0^1(D\cF_{v+\e h}(h)-D\cF_v(h)) d\e\\
\end{split}
\eee
Hence
\bee
||\cF(v+h)-\cF(v)-D\cF_v(h)||_{C^{\sigma,\frac\sigma2}}\le C||h||^2.
\eee
So $D\cF_v$ is the differential of $\cF$ at $v$. This completes the proof of the lemma.
\end{proof}

\begin{lma}\label{l-bijection}
In addition to Assumption \ref{a-1}, suppose $|||\on^k\cH|||_{\Omega}$ are uniformly bounded. For any $v\in \mathcal{B}$, $D\cF_v$ is a bijection from $C_0^{2+\sigma,1+\frac\sigma2}(M\times[0,s_0])$ onto
$C^{ \sigma, \frac\sigma2}(M\times[0,s_0]$.
\end{lma}
\begin{proof} Again, denote $C_0^{2+\sigma,1+\frac\sigma2}(M\times[0,s_0])$ by $C^{2+\sigma,1+\frac\sigma2}_0$ etc.  To prove that $D\cF_v$ is injective, let $\phi\in C_0^{2+\sigma,1+\frac\sigma2}$ such that $D\cF(\phi)=0$. Since $(M,g(0))$ has bounded curvature, we can find a smooth function $\eta$ on $M$ such that $\a$ is uniformly equivalent to the distance function of $g(0)$ with respect to a fixed point such that $\eta$ has bounded gradient and bounded Hessian with respect to $g(0)$, see \cite{Shi97,Tam2010}. Hence $\eta_i, \eta_{ij}$ with respect to any quasi local coordinates are uniformly bounded by a constant $C_1$ independent of the coordinate neighborhood. Here we denote the pull back of a function $f$ by $f$ again. Using $\eta$ one can proceed in a standard way. Namely in any quasi local coordinate neighborhood, we have
\bee
|D\cF_v(\exp(\a))|\le C_2\exp(\a)
\eee
where $C_2$ is independent of the quasi local coordinate neighborhood. Let  $\Theta=  \exp(C_3s+\a)$ for some $C_3$ to be determined later. Then for $\e>0$, $\phi-\e\Theta\le0$ at $s=0$. Suppose it is positive somewhere, then there is $p\in M$ and $0<s\le s_0$ such that $\phi-\e\Theta $ attains maximum at $(p,s)$. In the quasi coordinate neighborhood centered at $p$, we have, $\phi_i=\e\Theta_i$ and
\bee
\begin{split}
0\le &(\phi_s-\e\Theta)_s-\,^v a^{ij}(\phi-\e\Theta)_{ij}\\
=&- \,^vc(\phi-\e\Theta)-D\cF_v(\e\Theta)\\
\le &\e\Theta(-C_3 +C_2)\\
=&-\e C_2\Theta
\end{split}
\eee
if $C_3=2C_2$. This is a contradiction. Hence $\phi\le \e\Theta$. Let $\e\to0$ we have $\phi\le0$. Similarly, $-\phi\le0$. So $\phi=0$ and $D\cF_v$ is injective.

To prove that $\cF_v$ is surjective, let $f\in C^{\sigma,\frac\sigma2}$ and let $\Omega_\ell$ be compact domains in $M$ with smooth boundary such that $\Omega_{\ell}\Subset \Omega_{\ell+1}$ which exhaust $M$.  Let $\eta_\ell$ be smooth function so that $0\le \eta_\ell\le1$, $\eta_\ell=0$ outside $\Omega_{\ell+1}$ and $\eta_\ell=1$ on $\Omega_{\ell}$. Then one can find a solution   $\phi_\ell$  of the following initial-boundary value problem:
\bee
\left\{
  \begin{array}{ll}
    D\cF_v(\phi_\ell)=\eta_\ell f, & \hbox{in $\Omega_{\ell+1}\times[0,s_0]$;}\\
\phi_\ell=0,&\hbox{at $s=0$ and $\p\Omega_{\ell+1}\times[0,s_0]$}.
  \end{array}
\right.
\eee
The solution $\phi_\ell$ exists  by \cite[Theorem 4, Chapter 3]{Friedman}, because $\eta_\ell f=0$ on $\p\Omega_{\ell+1}$, so that  the compatibility condition $D\cF_v(\psi)=\eta_\ell f$ at $\p\Omega_{\ell+1}\times\{0\}$, where $\psi\equiv 0$, is satisfied.   Moreover, the pull back of $\phi_{\ell}$ is in $C^{2+\sigma,1+\frac\sigma2}(D(r)\times[0,s_0])$ in any quasi local coordinate neighborhood contained in $\Omega_{\ell}$.    Since $\,^vc$ is uniformly bounded independent of quasi local coordinates, by the maximum principle   \cite[Theorem 2.10]{Lieberman} we can see $\phi_\ell$ is uniformly bounded by a constant  independent of $\ell$.      Standard Schauder estimates (see \cite[Theorem 4, Chapter 4]{Friedman})  show that we can find $\phi_\ell$ satisfying $D\cF(\phi_\ell)=f$ in $\Omega_{\ell}\times[0,s_0]$ such  that for any $p\in \Omega_\ell$ with $B_p(r)\Subset \Omega_\ell$, for   the quasi local coordinates neighborhood $\xi_p:D(r)\to B_p(r)$,  $||\phi_{\ell}||_{C^{2+\sigma,1+\frac\sigma2}(D(\frac r2)\times[0,s_0])}$ is bounded by a constant independent of $p$ and $\ell$. Here we denote the pull back of $\phi_\ell$ also by $\phi_\ell$. Passing to a subsequence, we can find $\phi$ such that $D\cF_v(\phi)=f$ on $M\times[0,s_0]$ with $\phi=0$ at $s=0$. Moreover, $||\phi ||_{C^{2+\sigma,1+\frac\sigma2}(D(\frac r2)\times[0,s_0])}$ uniformly bounded by a constant in any quasi local coordinate neighborhood $\xi_p:D(r)\to B_p(r)$. From this we see that $\phi\in C^{2+\sigma,1+\frac\sigma2}(M\times[0,s_0])$ by Lemma \ref{l-equiv-holder}.

\end{proof}

Now we are ready to prove the following short time existence result.

\begin{proof}[Proof of Theorem \ref{t-shorttime Lorentz}] By Lemmas \ref{l-tilt},  \ref{l-foliation-1}, we may assume $a>0$ is small enough so that $\bX$ is an immersion. In order to construct the flow, without loss of generality we may identify $N$ with $(-a,a)\times M$ with metric $G=-dt^2+g(t)$. Moreover, we can introduce a family of quasi local coordinate neighborhoods as described by Lemma \ref{l-bounded geometry} so  that $g(t)$ satisfy the conclusions of  Lemma \ref{l-foliation-2}.

Let $v_0(\bx,s)=sH^o(\bx,0)-s\cH(\bx,0)$ defined on $M\times[0,s_0]$, where $H^o$ is the mean curvature of the slice $t=0$ in $(-a,a)\times M$. By Lemmas \ref{l-foliation-2},  \ref{l-tilt}, \ref{l-tilt-equivalent} and the fact that $|||\on^k\cH|||_{N}$ is finite  for $k=0, 1$, we can choose $s_0>0$ small enough so that $\sup_{M\times[0,s_0]}|v_0|<a$ and $\inf_{M\times[0,s_0]}(1-|\n v_0|^2)>\frac12$. On the other hand, fixed $0<\sigma<1$,  one can define $C^{2\sigma,\frac\sigma2}_0(M\times[0,s_0])$ and $C^{\sigma,\frac\sigma2}(M\times[0,s_0])$. As before let
$$
\mathcal{B}=\{v\in C^{2+\sigma,1+\frac\sigma2}_0(M\times[0,s_0])| \sup_{M\times[0,s_0]}|v|<a, \inf_{M\times[0,s_0]}(1-|\n v|^2)>\frac12\}.
$$
Hence $v_0\in \mathcal{B}.$ In fact, by Lemma \ref{l-foliation-2} and the fact that $|||\on^k\cH|||_{N}$ are finite  for any $k$, we have
\be\label{e-Lip}
|\xi_p^*(v_0)(\bx, s)-\xi_p^*(v_0)(\bx', s')|\le C_1\lf(|\bx-\bx'|+|s-s'|\ri)
\ee
in $D(r)\times[0,s_0]$, for   any quasi local coordinate neighborhood $\xi_p^*:D(r)\to B_p(r)$. Here $C_1$ is a constant independent of $p$.

 Let
\bee
\cF(v)=
\frac{\p v}{\p s}+\la \ppt,\vn\ra^{-1}(H-\cH)
\eee
where $\vn$ is the future direct unit normal of the graph of $v$: $ (\bx, v(\bx,s))\in (-a,a)\times M$ and $H$ is the mean curvature. Then $\cF$ maps $\mathcal{B}$ to $C^{\sigma,\frac\sigma2}(M\times[0,s_0]$ which is $C^1$ and $D\cF_v$ is bijective by Lemmas \ref{l-F-range} and \ref{l-bijection}. Hence $\cF$ maps a neighborhood of $v_0$ onto a neighborhood of $f_0=:\cF(v_0)$. By the definition of $v_0$ and \eqref{e-F-2}, $f_0(\bx,0)=0$ because $H^o=\frac12 g^{ij}\p_t g_{ij}$.

For $0<\lambda<s_0$, define
\bee
f_\lambda(\bx,s)=\left\{
  \begin{array}{ll}
    0 & \hbox{if $0\le s\le \lambda$.}\\
f_0(\bx, s-\lambda)& \hbox{if $\lambda< s\le s_0$.}
  \end{array}
\right.
\eee

\begin{claim}\label{claim}
$||f_0-f_\lambda||_{C^{\sigma,\frac\sigma2}(M\times[0,s_0])}\le C\lambda^{1-\sigma}$ for some $C$ independent of $\lambda$.
\end{claim}
If the claim is true, then $f_\lambda\in C^{\sigma,\frac\sigma2}(M\times[0,s_0])$ and by the inverse function theorem, for $\lambda>0$ small enough, we can find $w$ such that
$$
\mathcal{F}(w)=f_\lambda.
$$
In particular, $w$ is a solution of \eqref{e-w-1} on $M_0\times[0,\lambda]$, $w\in \mathcal{B}$.

To prove {\bf Claim} \ref{claim}, let $\xi_p:D(r)\to B_p(r)$ be a quasi local coordinate neighborhood. Since  $f_0=0$, at $s=0$, by \eqref{e-Lip},   that there is $C_2>0$ independent of $p$, $\lambda$ such that if $x, x'\in D(r)$ and $s, s'\in [0,s_0]$, we have
\be\label{e-wwtau}
|\xi_p^* f_\lambda(\bx,s)-\xi_p^*f_\lambda(\bx',s')|\le C_2(|\bx-\bx'|+|s-s'|).
\ee
Let $\eta=f_0-f_\lambda$. Then
\bee
\eta(\bx,s)=\left\{
  \begin{array}{ll}
    f_0(\bx,s) & \hbox{if $0\le s\le \lambda$.}\\
f_0(\bx, s)-f_0(\bx,s-\lambda)& \hbox{if $\lambda<s\le s_0$.}
  \end{array}
\right.
\eee
There is a constant $C_3>C_2$ independent of $p$ such that $|\eta|\le C_3\lambda$ by \eqref{e-Lip} and the fact that $f_0=0$ at $s=0$.
For $\bx, \bx'\in D(r), s, s'\in[0,s_0]$.

\underline{\it Case 1}: $s, s'\le \lambda$. Then
\bee
|\xi_p^*\eta(\bx,s)-\xi_p^*\eta(\bx',s')|=|\xi_p^*f_0(\bx,s)-\xi_p^*f_0(\bx',s')|\le C_3\min\{\lambda, |x-x'|+|s-s'|\}
\eee
because $f_0(\bx,0)=0$ and $s, s'\le \lambda$. If $|\bx-\bx'|\ge \lambda$, then
\bee
|\xi_p^*\eta(\bx,s)-\xi_p^*\eta(\bx',s')|\le C_3\lambda^{1-\sigma}\lambda^\sigma\le C_3\lambda^{1-\sigma}(|\bx-\bx'|^\sigma+|s-s'|^{\frac\sigma2}).
\eee
If $|\bx-\bx'|\le \lambda$, then by \eqref{e-Lip}

\bee
\begin{split}
|\xi_p^*\eta(\bx,s)-\xi_p^*\eta(\bx',s')|\le &C_2(|\bx-\bx'|+|s-s'|)\\
=&C_2(|\bx-\bx'|^{1-\sigma}|\bx-\bx'|^\sigma+|s-s'|^{1-\frac\sigma2}|s-s'|^\frac\sigma2)\\
\le &C_4\lambda^{1-\sigma}(|\bx-\bx'|^\sigma+|s-s'|^{\frac\sigma2}).
\end{split}
\eee
for some $C_4$ independent of $p$.

\underline{\it Case 2}: $s, s'\ge \lambda$.
\bee
\begin{split}
|\xi_p^*\eta(\bx,s)-\xi_p^*\eta(\bx',s')|=&|\xi_p^*f_0(\bx,s)-\xi_p^*f_0(\bx,s-\lambda)
-\xi_p^*f_0(\bx',s')
+\xi_p^*f_0(\bx',s'-\lambda|\\
\le& C_3\min\{\lambda, |\bx-\bx'|+|s-s'|\}
\end{split}
\eee
As before, we have
\bee
|\xi_p^*\eta(\bx,s)-\xi_p^*\eta(\bx',s')|\le C_4\lambda^{1-\sigma}(|x-x'|^\sigma+|s-s'|^{\frac\sigma2}).
\eee

{\it Case 3}: $s\ge \lambda\ge s'$, then
\bee
\begin{split}
|\xi_p^*\eta(\bx,s)-\xi_p^*\eta(\bx',s')|\le &|\xi_p^*\eta(\bx,s)-\xi_p^*\eta(\bx,\lambda)|+|\xi_p^*\eta(\bx,\lambda)-\xi_p^*\eta(\bx',s')|\\
\le &2C_4\lambda^{1-\sigma}\lf(|s-\lambda|^\frac\sigma 2+|\bx-\bx'|^\sigma+|s'-\lambda|^\frac\sigma2\ri)\\
\le &3C_4\lambda^{1-\sigma}\lf( |\bx-\bx'|^\sigma+|s-s'|^\frac\sigma2\ri)
\end{split}
\eee
because $s\ge \lambda\ge s'$. Hence the claim is true.

To conclude, we can find a solution $w$ to \eqref{e-w-1}  on $M\times[0,\e]$ for some $\e>0$, so that   $w\in \mathcal{B}$. By differentiate the equation \eqref{e-F-2} with respect to $x^k $, by Schauder  estimates, one can conclude that $\xi_p^*(w)_i$ the $C^{2+\sigma,1+\frac\sigma2}(D(r')\times[0,\e])$ are uniformly bounded for $r'<r$, in any quasi local coordinate neighborhood. Here we have used the fact that all derivatives of $g_{ij}(t)$ and all derivatives of $\cH$ are bounded and the fact that $1-|\n w|^2>\frac12$. Continue in this way, we see that all the derivatives of $\xi_p^*(w)$ with respect to $\bx$ are bounded. Using the equation \eqref{e-F-2} again, we conclude that all the derivatives of $\xi^*(w_s)$ with respect to $\bx$ are bounded. Differentiate with respect to $s$, one can finally conclude that all the derivatives of $\xi^*_p(w)$ in $D(\frac r2)$ are uniformly bounded independent of $p$. By Lemma \ref{l-reduction-2}, we conclude that the short time existence of the prescribed mean curvature flow. Moreover, $\bF$ is smooth and the second fundamental form of $\cF(\cdot,s)(M)$ together with it covariant derivatives are uniformly bounded by the bounds of derivatives of $w$ and the bounds of derivatives of $g(t)$. Since $1-|\n w|^2>\frac12$, we have
$$
-\la \ppt, \vn\ra\le C
$$
for some constant $C$ in $M\times[0,\lambda]$. Hence by \cite{Bartnik1984}
\bee
-\la \bT,\vn\ra\le 2\la   \bT,\ppt\ra\la \ppt,\vn\ra\le C'
\eee
for some constant $C'$ in $M\times[0,\lambda]$. Since $1-|\n w|^2>\frac12$, by Lemma \ref{l-foliation-2}, the metric pulled back by $\cF(\cdot,s)$ is equivalent to $g(0)$ and  is complete. Hence the theorem is proved.

\end{proof}

We would like to discuss the condition (i) in Theorem \ref{t-shorttime Lorentz}. If $N$ is time like geodesically complete, then (i) is satisfied automatically. But this is too strong. On the other hand, we have the following.

\begin{lma}\label{e-existence geodesic} Let $(N,G)$ be a Lorentz manifold with time function $\tau$ as before. Suppose there is a smooth function $\rho>0$ satisfying condition \eqref{e-rho}. Let $p\in \Omega_{\tau_1,\tau_2}\subset N$ and let $v$ be a future time like unit vector at $p$. Assume $|||\on\bT|||\le c$ in $\Omega_{\tau_1,\tau_2}$. Suppose either $|||1/\a|||_{0,\Omega_{\tau_1,\tau_2}}\le b$, or $|||\on \log\a|||_{0,\Omega_{\tau_1,\tau_2}}\le b$. Then the time like geodesic $\gamma(t)$ with $\gamma(0)=p, \gamma'(0)=v$  can be defined on $(-t_0,t_0)$ for some $t_0$ depending only on $b$, $-\la v,\bT\ra$, $\tau_2-\tau(p), \tau(p)-\tau_1$ and upper bound of $1/(\a(p)) $ so that $\gamma(t)\in \Omega_{\tau_1,\tau_2}$.
\end{lma}
\begin{proof} Let $a>0$ be given by Lemma \ref{l-tilt} depending only on $c, \kappa_0=-\la v,\bT\ra$. Suppose we can prove the following a priori bound. There is a constant $C$ depending only on the quantities mentioned in the lemma such that whenever $\gamma(t)\subset \Omega_{\tau_1,\tau_2}$, $t\in (-t',t')$ with $t'\le a$, then $\a(\gamma(t))\ge C^{-1}$. Then the lemma is true with
$$
t_0=\min\{a, (\tau(p)-\tau_1)/(2C\kappa_0), (\tau_2-\tau(p))/(2C\kappa_0).
$$
In fact if $t_1$ is the supremum of $t$ so that $\gamma(t)$ is defined and $\gamma(-t,t)\subset\Omega_{\tau_1,\tau_2}$.
If $t_1<t_0$, by Lemma \ref{l-tilt},   $\kappa=:-\la\gamma'(t),\bT\ra\le 2\kappa_0$.
$$
\tau(\gamma(t))-\tau(p)=\int_0^t\la \gamma',\on \tau\ra d\sigma=\int_0^t-\frac1\a\la \gamma',\bT\ra d\sigma,
$$
and
$$
|\tau(\gamma(t))-\tau(p)|\le 2C\kappa_0|t|.
$$
Hence
\be\label{e-tau bound}
   \tau_1+\e\le (\tau(p)-\tau_1)-2C\kappa_0|t|+\tau_1\le \tau(\gamma(t))\le  - (\tau_2-\tau(p))+2C\kappa_0|t|+\tau_2\le \tau_2-\e
\ee
for all $t$ for some $\e>0$. On the other hand,
\bee
|\rho(\gamma(t))-\rho(p)|=|\int_0^t\la\gamma',\on \rho\ra d\sigma|\le C_1 |t|.
\eee
for some constant $C_1$ depending only on $\kappa_0, |||\on\rho|||_{0,\Omega_{\tau_1,\tau_2}}$. In particular,
\be
\rho(\gamma(t))\le C_2
\ee
on $(-t_1,t_1)$. For $t,t'\uparrow t_1$,
\bee
d_{G_E}(\gamma(t),\gamma(t'))\le |\int_t^{t'}|||\gamma'|||d\sigma\le 2\kappa_0|t-t'|.
\eee
By the compactness of the set $\ol\Omega_{\tau_1,\tau_2}\cap\{\rho\le C_2\}$, we conclude that $\gamma(t)$ can be extended up to $-t_1,t_1$. By \eqref{e-tau bound}, we must have $t_1\ge t_0$.

It remains to obtain the a priori estimate mentioned above. In case $|||1/\a|||_{0,\Omega_{\tau_1,\tau_2}}$ this is obvious. Suppose $|||\on\log \a|||_{0,\Omega_{\tau_1,\tau_2}}$, then as long as $\gamma \subset \Omega_{\tau_1,\tau_2}$, then
\bee
\log(\a(\gamma(t)))-\log\a(p)=\int_0^t\la \gamma',\on\log \a\ra d\sigma.
\eee
From this one can obtain the required estimate.
\end{proof}

\begin{rem}\label{r-short time} Hence if in Theorem \ref{t-shorttime Lorentz}, condition (i) is replaced by  $\tau_1<\inf_M\tau\circ \bX_0\le \sup_M\tau\circ\bX_0<\tau_2$, and either $|||1/\a|||_{0,\Omega_{\tau_1,\tau_2}}<\infty$ or $\sup_M1/(\a\circ\bX_0)<\infty $ and $|||\on\log\a|||_{0,\Omega_{\tau_1,\tau_2}}<\infty$, then the conclusion of the theorem is still true. Here conditions involving $\Omega$ will be replaced by $\Omega_{\tau_1,\tau_2}$.

\end{rem}

\section{Long time existence}\label{s-long time}
In this section, we will prove the long time existence result:  Theorem  \ref{t-longtime}.
Let us begin with the following lemma.

\begin{lma}\label{l-extension} Let $(N,G)$, $\bX_0$, $\tau$, $\rho$ be as in Theorem \ref{t-longtime}. Suppose $\bF:M\times[0,s_0)$, $s_0<\infty$,  is a solution of \eqref{e-mcf-1} so that the following are true:
\begin{enumerate}
  \item[(i)] For all $s<s_0$, $\sup_{M\times[0,s]}\kappa<\infty$ and $\sup_{M\times[0,s]}|\n^k A|<\infty$   for all $k\ge 0$ .
  \item[(ii)]  The height function $u(\bx, s)=:\tau(\bF(\bx,s))$   satisfies $$\tau_1<\inf_{M\times[0,s_0)}u\le \sup_{M\times[0,s_0)}u<  \tau_2$$ for some $\tau_-<\tau_1<\tau_2<\tau_+$.

\end{enumerate}
 Then $\bF$ can be extended smoothly up to $s=s_0$ as a solution of the prescribed mean curvature flow so that the tilt factor and all the covariant derivatives of the second fundamental form  of $\bF^{s_0}(M)$ are uniformly bounded.

\end{lma}

Assume the lemma is true, let us prove Theorem \ref{t-longtime}.
\begin{proof} By Theorem \ref{t-shorttime Lorentz} and by Remark \ref{r-short time}, there is $\e>0$ such that the prescribed mean curvature flow \eqref{e-mcf-1} has a solution $\bF$ on $M\times[0,\e]$ such that $\sup_{M\times[0,\e]}|\n^k A|<\infty$ for all $k\ge 0$, and $\sup_{M\times[0,\e]}\kappa<\infty$. Let $s_{\max}$ be as in the theorem and $s_{\max}<\infty$. Suppose
$$
\tau_-<\tau_1<\inf_{M\times[0,s_{\max})}\tau(\bF)\le \sup_{M\times[0,s_{\max})}\tau(\bF)<\tau_2<\tau_+.
$$
for some $\tau_1, \tau_2$. By Lemma \ref{l-extension}, $\bF$ can be extended smoothly as a solution to the prescribed mean curvature flow up to $s_{\max}$ so that all the derivatives of the second fundamental form and the tilt factor of $\bF(\cdot, s_{\max})$ are uniformly bounded. In particular, $\tau(\bF(\bx,s ))\to \tau(\bF(\bx,s_{\max}))$ as $s\uparrow s_{\max} $. Hence $\tau_1\le \tau(\bF(\bx,s_{\max}))\le \tau_2$ for all $\bx\in M$.
On the other hand,
\bee
\begin{split}
\log \a(\bF(\bx,s_{\max}))-\log \a(\bF(\bx,0))=&\int_0^{s_{\max}}\la \p_s\bF, \on \log\a \ra d\sigma\\
=&\int_0^{s_{\max}}\la (H-\cH)\vn, \on \log\a \ra d\sigma.
\end{split}
\eee
Since $|||\on\log \a||| $, $H-\cH$, and the tilt factor $\kappa$ are all uniformly bounded on $\Omega_{\tau_1,\tau_2}$, we have
\bee
\log \a(\bF(\bx,s_{\max}))\ge C\log \a(\bF(\bx,0))
\eee
for some $C>0$ independent of $\bx$. Since $\inf_M\a(\bF(\cdot,0))>0$, we conclude that
$\inf_M\a(\bF(\cdot,s_{\max}))>0$.  By Theorem \ref{t-shorttime Lorentz} and by Remark \ref{r-short time},  $\bF$ can be extended beyond $s_{\max}$ satisfying the conditions in Theorem \ref{t-longtime}. This contradicts the definition of $s_{\max}$. This completes the proof of the theorem.

\end{proof}

It remains to prove Lemma \ref{l-extension}. This follows from   the basic estimates obtained by Ecker and Husiken \cite{EckerHusiken1991,Ecker1993} and standard arguments as in \cite[Chapter 6]{AndrewsChowGuentherLangford}. Here we use the existence of $\rho$ to obtain $C^0$ convergence. For completeness, we sketch the proof.  First  recall the evolution equations of geometric quantities along the flow, see \cite{EckerHusiken1991,Ecker1993}:
 Let $\bF$ satisfies \eqref{e-mcf-1}.

\begin{lma}\label{l-evolution-eqs}
With the notation as in Theorem \ref{t-longtime}, let $g(s)$ be the induced metric on $M_s$ with unit future pointing unit normal $\vn$,
second fundamental form $A$, mean curvature $H$, height function $u$, tilt factor $\kappa$.
\begin{enumerate}
  \item[\bf(i)]
  $
  \displaystyle{\frac{\p}{\p s}g_{ij}=2 (H-\mathcal{H})A_{ij}.}
  $
  \item[\bf(ii)] $
\displaystyle{\pps g^{ij}=-2 (H-\mathcal{H})A^{ij}}.$

\item[\bf(iii)]
\bee
\begin{split}
\pps \Gamma_{ij}^k=&- g^{kl}\bigg[(H-\mathcal{H})(A_{jl;i}+A_{il;j}-A_{ij;l})\\
&+(H-\mathcal{H})_iA_{jl}
 +(H-\mathcal{H})_jA_{il}-(H-\mathcal{H})_lA_{ij}\bigg].
 \end{split}
 \eee
 Here $;$ is the covariant derivative of $g(s)$ and $\Gamma_{ij}^k$ is the connection.

 \item[\bf(iv)] $\dps{\pps \vn=    \nabla (H-\mathcal{H})}$.

 \item[\bf(v)]

\bee
\begin{split}
 (\pps-\Delta)A_{ij}
=&- (|A|^2+\ol{\Ric}(\vn,\vn))A_{ij}+2 HA_{ri}A^r_j\\
 &+2g^{kl}A_k^q\oR_{qijl}-g^{kl}A_j^q \oR_{qkli} -g^{kl}A_l^q\oR_{qikj}
 \\&+g^{kl}\on_l\oR_{0ijk}
-g^{kl}\on_j\oR_{0kli}\\
&-\cH(g^{kl}A_{ik}A_{jl}-\oR_{0ij0})
 -\cH_{;ij}.
\end{split}
\eee
Here  the index 0 in $\oR$ corresponds to $\vn$.
\item[\bf(vi)]
\bee
 (\pps-\Delta)(H-\cH)=-\lf(|A|^2+\ol{\Ric}(\vn,\vn)+\la \on\cH,\vn\ra\ri)(H-\cH).
 \eee
\item [\bf(vii)] $\pps u=\a^{-1}\kappa(H-\cH)$, $\Delta  u=\div (\on\tau)+\a^{-1}\kappa H$ and
  \bee
  (\pps-\Delta)u=-\cH \a^{-1}\kappa -\div\on \tau.
  \eee
  Here the divergence is on the immersed surface $M_s$.
  \item[\bf(viii)]
  \bee
(\pps -\Delta)\kappa = -\kappa(|A|^2+\Ric(\vn,\vn)) -\bT(H_\bT) -(H-\cH)\la \on_\vn \bT,\vn\ra+\la \bT,\n^M \cH \ra.
\eee
Here $\bT(H_\bT)$ is the variation of the mean curvature under the deformation of $\bT$. (See \cite[2.10]{Bartnik1984}).
\end{enumerate}

\end{lma}

By \cite{Bartnik1984}, we have the following facts:
\begin{lma}\label{l-grad kappa}
Let $M$ be a spacelike hypersurface with tilt factor $\kappa$ and let $u=\tau|_M$. Then
\bee
|\n^Mu|^2=\alpha^{-2}(\kappa^2-1).
\eee
For any $ 1>\delta>0$,
\bee
|\n_M\kappa|^2\le (1+\frac\delta n )^{-1}\kappa^2|A|^2+C(\kappa^4+\kappa^2H^2),
 \eee
 for some $C=C(|||\on\bT|||, n, (1-\delta)^{-1})$.

\end{lma}

Since all the covariant derivatives of the curvature of $(M,g_0)$ are bounded, we can find a smooth function $\eta$ with bounded gradient and  bounded Hessian, such that $ d(\bx)+C\le \eta(\bx)\le 2d(\bx)+C'$ where $d(\bx)$ is the distance function with respect to $g_0$ from a fixed point $\bx_0$, see  \cite{Shi97,Tam2010}.

 By Lemma \ref{l-evolution-eqs} (i)--(iii) we have:
 \begin{lma}\label{l-exhaustion} Let $\bF$ be a solution of \eqref{e-mcf-1} on $M\times[0,s_0]$.  Suppose $|A|, |\n A|$ are uniformly bounded on $M\times[0,s_0]$, then
  gradient and the Hessian of $\eta$ with respect to $g(s)$ are uniformly bounded.
 \end{lma}
 Using $\eta$, we have the following maximum principle.
 \begin{lma}\label{l-max} Let $\bF$  and $\eta$ be as in Lemma \ref{l-exhaustion}. Suppose $f$ is a smooth function on $M\times[0,s_0]$ such that  $f\le C_1e^{C_2 \eta}$  for some $C_1, C_2>0$ and  $(\pps-\Delta)f\le 0$ whenever $f>0$. Assume $\sup_{M_0}f<\infty$. Then $\sup_{M\times[0,s_0]}f\le \max\{0,\sup_{M_0}f\}.$

 \end{lma}
 \begin{proof} Let $a=\sup_{M_0}f$ and let $\phi=e^{\sigma s+2C_2\eta}$ where $\sigma$ is to be determined. By Lemma \ref{l-exhaustion}, we have
 \bee
 (\pps-\Delta)\phi\ge (\sigma-C_3)\phi
 \eee
 for some $C_3>0$ in $M\times[0,s_0]$. Choose $\sigma=2C_3$, we have $(\pps-\Delta)\phi>0$. For any $\e>0$,
 \bee
 (\pps-\Delta)(f-a-\e\phi)<0
 \eee
 whenever $f>0$.   $\sup_{M\times[0,s_0]}(f-a-\e\phi)$ will be attained at some point in $(\bx,s)\in M\times[0,s_0]$. If $s>0$, we must have $f\le0$ at this point. So $\sup_{M\times[0,s_0]}(f-a-\e\phi)\le \max\{-a, 0\}$.  Let $\e\to0$, the result follows.
 \end{proof}

The following is a basic estimate obtained by Ecker \cite{Ecker1993}. We modify the proof a little bit using the idea in \cite{Bartnik1984}. The assumptions are slightly different from those in \cite{Ecker1993}. We give details of the proof in order to understand what conditions are sufficient.
\begin{lma}\label{l-grad}  [Ecker's basic estimate]
Let $\bF$ be a solution to the prescribed mean curvature flow \eqref{e-mcf-1}  such that  $\tau_-<\tau_1< \tau(\bF)<\tau_2<\tau_+$.

\begin{enumerate}
  \item [(i)] $\sup_{M\times[0,s_0]}(\kappa+|A|+|\n A|)<\infty$.

  \item[(ii)] $|||\n \bT|||_{1,\Omega_{\tau_1,\tau_2}}$, $ |||\cH|||_{1,\Omega}$, $|||\oRm|||_{1,\Omega_{\tau_1,\tau_2}}$, $|||\on \log\a|||_{1,\Omega_{\tau_1,\tau_2}}$, $|||\a|||_{1,\Omega_{\tau_1,\tau_2}}$ are uniformly bounded by $c_1$.
      \item[(iii)] $\ol\Ric(w,w)\ge -c_2$ for some $c_2\ge0$ for all time like unit vector $w$ in $\Omega_{\tau_1,\tau_2}$.
\end{enumerate}
Then there exist $\lambda>0, \mu>0, C$ depending only on $c_1, c_2, \tau_1, \tau_2, n$ with $C$ also depending on $\sup_{M_0}(\kappa+|H-\cH|)$ such that
\bee
\sup_{M\times[0,s_0]}(e^{\lambda u}\kappa^2+\mu(H-\cH)^2)\le C.
\eee
\end{lma}
\begin{proof}
Let $u=\tau(\bF)$. Then $\tau_1<u<\tau_2$.   By Lemma  \ref{l-evolution-eqs}(viii)

\be\label{e-expu}
\begin{split}
(\pps-\Delta)e^{\lambda u}
=&\lambda e^{\lambda u}\lf(-\cH\alpha^{-1}\kappa-\div_M\on\tau  -\lambda |\n u|^2\ri)\\
\le &\lambda e^{\lambda u}\lf(-\cH\alpha^{-1}\kappa+\kappa^2|||\on(\alpha^{-1}\bT) |||  -\lambda |\n u|^2\ri)\\
\le&\lambda e^{\lambda u}\lf(C_1\alpha^{-1}\kappa^2-\lambda |\n u|^2\ri)
\end{split}
\ee
for some $C_1=C_1(c_1, n)$ because $\on \tau=-\a^{-1}\bT$, $\on \a^{-1}=-\a^{-1}\on \log \a$ and $\kappa\ge 1$.
Using the facts that
 \bee
| \la \on_\vn \bT,\vn\ra|\le 2\kappa^2|||\on\bT|||;\  |\la\bT,\n \cH\ra|\le 2\kappa|||\on \cH|||,
\eee
 and
 \bee
 |\bT(H_\bT)|\le C(n)\lf(|||\on^2\bT|||\kappa^3+|||\on \bT||| \kappa^2|A|\ri)
 \eee
 by Bartnik \cite[2.10]{Bartnik1984}, using Lemma \ref{l-evolution-eqs} we have
  for any $\e>0$:

\bee
\begin{split}
(\pps -\Delta)\kappa^2=&2\kappa(\pps-\Delta)\kappa-2|\n \kappa|^2\\
\le &-2(1-\e)\kappa^2|A|^2+2C_2\e^{-1}\kappa^4-2|\n\kappa|^2\\
\end{split}
\eee
for some $C_2=C_2(c_1, n)$.
On the other hand, by Lemma \ref{l-grad kappa}, we have for $1>\delta>0$,
\bee
|\n\kappa|^2\le (1+\frac\delta n)^{-1}\kappa^2|A|^2+C(\kappa^4+\kappa^2H^2)
\eee
for some $C=C(|||\on T|||_{0,\Omega}, n, (1-\delta)^{-1})$. Hence by choosing suitable $\delta$,  $\e$, we have
\be\label{e-kappa2}
\begin{split}
(\pps -\Delta)\kappa^2
\le &-4(1+\frac1{4n} ) |\n\kappa|^2 +C_3(\kappa^4+\kappa^2H^2)\\
\end{split}
\ee
for some $C_3=C_3(c_1,n)$.
Since $\ol{\Ric}(\vn,\vn)\ge-c_2$, and
\be\label{e-A-H}
|A|^2\ge \frac1nH^2\ge \frac1{2n}(H-\cH)^2-\frac1n\cH^2,
 \ee
 by Lemma   \ref{l-evolution-eqs}(vi), we have
\be\label{e-HcH}
\begin{split}
(\pps -\Delta)(H-\cH)^2
\le&-2(|A|^2-C_4\kappa)(H-\cH)^2-2|\n(H-\cH)|^2\\
\le &-\frac1{n}(H-\cH)^4+C_5\kappa(H-\cH)^2+\frac2n\cH^2-2|\n(H-\cH)|^2\\
\le &   -\frac1{2n}(H-\cH)^4- 2|\n(H-\cH)|^2+C_6\kappa^2
\end{split}
\ee
for some $C_4, C_5, C_6>0$ depending only on  $c_1,c_2, n$. Let $\mu>0$ to be determined, and let
\bee
\Phi=e^{\lambda u}\kappa^2+\mu(H-\cH)^2.
\eee

We may assume $C_3\ge C_1$. Then  by combining \eqref{e-expu}, \eqref{e-kappa2}, \eqref{e-HcH}, we have:

\be\label{e-Phi}
\begin{split}
 (\pps -\Delta)\Phi\le&  e^{\lambda u}\lf[-4(1+\frac1{4n} ) |\n\kappa|^2-\lambda^2\kappa^2|\n u|^2\ri]
+C_3e^{\lambda u}\kappa^2\lf((\lambda\a^{-1}+1)\kappa^2+H^2 \ri) \\
&-2\la\n e^{\lambda u},\n \kappa^2\ra- \frac\mu{2n}(H-\cH)^4- 2\mu |\n(H-\cH)|^2+C_6\mu\kappa^2
\end{split}
\ee
Now
\bee
\begin{split}
&e^{\lambda u}\lf[-4(1+\frac1{4n} ) |\n\kappa|^2-\lambda^2\kappa^2|\n u|^2\ri]
-2\la\n e^{\lambda u},\n \kappa^2\ra\\
\le&e^{\lambda u}\lf[- \frac{1+4n}{4n}  \kappa^{-2}|\n\kappa^2|^2-\lambda^2\kappa^2|\n u|^2
+2\lambda |\n u||\n\kappa^2|\ri] \\
=&e^{\lambda u}\lf[- \frac{1+4n}{4n}  \kappa^{-2}|\n\kappa^2|^2-\frac{4n}{1+4n}\lambda^2\kappa^2|\n u|^2
+2\lambda |\n u||\n\kappa^2|\ri]-\frac1{1+4n}\lambda^2e^{\lambda u}\kappa^2|\n u|^2\\
\le & -\frac1{1+4n}\lambda^2e^{\lambda u}\a^{-2}\kappa^2(\kappa^2-1)
\end{split}
\eee
by Lemma \ref{l-grad kappa}.
\bee
\begin{split}
 (\pps -\Delta)\Phi\le&  -\frac1{4n+1}\lambda^2\kappa^2e^{\lambda u}\a^{-2}(\kappa^2-1)+C_3e^{\lambda u}\kappa^2\lf((\lambda\a^{-1}+1)\kappa^2+H^2 \ri)\\
 &  - \frac\mu{2n}(H-\cH)^4- 2\mu |\n(H-\cH)|^2+C_6\mu\kappa^2.\\
\end{split}
\eee
Since   $H^2\le 2(H-\cH)^2+2\cH^2$,
\bee
e^{\lambda u}\kappa^2H^2\le  e^{\lambda u}(\kappa^4+(H-\cH)^4+2\cH^2\kappa^2).
\eee
Hence
\bee
\begin{split}
 (\pps -\Delta)\Phi\le& -\frac1{4n+1}\lambda^2\kappa^2e^{\lambda u}\a^{-2}(\kappa^2-1)+C_7e^{\lambda u}\lf(\lambda \a^{-1} +1\ri)\kappa^4\\
 &- \lf(\frac\mu{2n}-C_3e^{\lambda u}\ri)(H-\cH)^4- 2\mu |\n(H-\cH)|^2+C_8e^{-\lambda u}(\mu+e^{\lambda u})^2
\end{split}
\eee
for some positive constants $C_7, C_8$ depending only on $c_1, c_2, n$.
Choose $\mu$ such that
$$
\frac\mu{2n}-C_3e^{\lambda \tau_2}=1.
$$
Note that  $\mu=\mu(c_1,c_2, n,\lambda, \tau_2)$.
We have
\be\label{e-Phi-kappa small}
\begin{split}
(\pps-\Delta)\Phi\le &-\frac1{4n+1}\lambda^2\kappa^2e^{\lambda u}\a^{-2}(\kappa^2-1)+C_7e^{\lambda u}\lf(\lambda \a^{-1} +1\ri)\kappa^4\\
 &-  (H-\cH)^4- 2\mu |\n(H-\cH)|^2+C_8e^{-\lambda u}(\mu+e^{\lambda u})^2.
 \end{split}
 \ee
 Choose $\lambda$ large enough depending only on $|||\a|||_{0,\Omega_{\tau_1,\tau_2}}$ and $C_7$, so that {\it if $\kappa\ge 2$}, then
$$
-\frac1{4n+1}\lambda^2  \a^{-2}(\kappa^2-1)+C_7 \lf(\lambda \a^{-1} +1\ri)\kappa^2\le - \kappa^2.
$$
Note that $\lambda=\lambda(c_1,c_2,n)$ and   $\mu=\mu(c_1,c_2, n,\tau_2)$.
So for $\kappa\ge 2$, we have

\be\label{e-Phi-kappa large}
\begin{split}
 (\pps -\Delta)\Phi\le&-e^{\lambda u}\kappa^4-(H-\cH)^4+ C_8e^{-\lambda u}(\mu+e^{\lambda u})^2
 \end{split}
\ee
As in the proof of Lemma \ref{l-max}, let $\phi=e^{\sigma s+\eta}$ for some $\sigma>0$  where $\eta$ is as in Lemma \ref{l-exhaustion} so that
\bee
(\pps-\Delta)\phi>0.
\eee
Let $m_0=\sup_{M_0}\Phi$. For any $\e>0$. Then $\Phi-m_0-\e \phi<0$ at $s=0$ and at infinity. If $\Phi-m_0-\e\phi>0$ somewhere, then there is $\bx_1\in M, 0<s_1\le s_0$ so that $(\Phi-m_0-\e\phi)(\bx_1,s_1)=
\sup_{M\times[0,s_0]}(\Phi-m_0-\e\phi)$. At this point we have:
\bee
0\le (\pps-\Delta)(\Phi-m_0-\e\phi)< (\pps-\Delta) \Phi.
\eee
Suppose at this point $\kappa\ge 2$, then by \eqref{e-Phi-kappa large}
\bee
0\le -e^{\lambda u}\kappa^4-(H-\cH)^4+ C_9e^{-\lambda u}(\mu+e^{\lambda u})^2.
\eee
Hence $e^{\lambda u}\kappa^2 +(H-\cH)^2\le C(c_1,c_2, n,\tau_1, \tau_2)$. This implies $\Phi-\e\phi\le C$ for some $C=C(c_1,c_2, n,\tau_1, \tau_2)$.

Suppose at this point $\kappa\le 2$, then by \eqref{e-Phi-kappa small}
\bee
(H-\cH)^4\le C(c_1,c_2,n,\tau_1,\tau_2).
\eee
and we also have $\Phi-\e\phi\le C(c_1,c_2,n,\tau_1,\tau_2)$. In any case, $$\Phi-\e\phi\le C(c_1,c_2,n,\tau_1,\tau_2)+m_0.$$ Let $\e\to0$, the result follows.

\end{proof}

By  \cite[p.604]{EckerHusiken1991} and the proof of  \cite[Th. 4.1]{Huisken1987}, we have the following estimate for $|\n^k A|$.

\begin{lma}\label{l-A-est} Let $\bF:M\times[0,s_0]$ be a solution to the prescribed mean curvature flow \eqref{e-mcf-1} so that the initial metric on $M_0$ is complete with bounded second fundamental form  and so that the tilt factor of the immersed surface $\bF^s: M\to N$ is uniformly bounded by $\kappa_0$. Suppose $\bF(M\times[0,s_0])\subset \Omega$ and $|||\oRm|||_{k,\Omega}$, $|||\cH|||_{k,\Omega}$ are finite for all $k\ge0$ and suppose $\sup_{M\times[0,s_0]}|\n^kA|<\infty$ for all $k$. Then for all $m\ge0$,
\bee
\sup_{M\times[0,s_0]}|\n^m A|^2\le C_m
\eee
for some $C_m$ depending only on $  n,m,$ $ \sup_{M\times[0,s_0]}\kappa$, $||| \oRm |||_{m+1,\Omega}$, $|||\cH|||_{m+2,\Omega}$, and
$\sum_{j=0}^{m-1}\sup_{M_0}|\n^jA|$.
\end{lma}

Now we are ready to prove Lemma \ref{l-extension}.

\begin{proof}[Proof of Lemma \ref{l-extension}] By Lemma  \ref{l-grad}, there is a constant $\kappa_1 $ such that
\be\label{e-kappa-bound}
\sup_{M\times[0,s_0)}\kappa\le \kappa_1.
\ee
Using this, by Lemma \ref{l-A-est}, for any $m$, there is a constant $c_m$ such that
\be\label{e-A-bound}
\sup_{M\times[0,s_0)}|\n^m A|\le c_m.
\ee
Let $\bx_0\in M$ and let $B_{\bx_0}(r_0)$ be a geodesic ball with respect to $g_0$ with center $\bx_0$ and radius $r_0$, which is inside  a coordinate chart with $\bx=(x^1,\cdots, x^n)$. First we want to obtain $C^0$ limit of $\bF(\cdot,s)$ as $s\to s_0$.  Let $\rho$ be the function satisfying condition \eqref{e-rho}. Since $\bF(\bx,s)\in \Omega_{\tau_1,\tau_2}$, by the assumption on $\rho$, and \eqref{e-mcf-1} we have for $\bx\in B_{\bx_0}(r_0)$,
\bee
\rho(\bF(\bx,s))=\int_0^s\la \on \rho, \p_s\bF\ra d\sigma\le C_1s
\eee
for some constant $C_1$ depending only on $\kappa_0$, $\sup_{M\times[0,s_0)}(|H-\cH|)$ and $ |||\on \rho|||_{0,\Omega_{\tau_1,\tau_2}}$. Hence $\bF(B_{\bx_0}(r_0)\times[0,s_0))\subset \ol\Omega_{\tau_1,\tau_2}\cap \{\rho\le \rho_0\}=K$ which is compact, where $$\rho_0=C_1s_0+\sup_{B_{\bx_0}(r_0)}\rho+1.$$
Hence $K$ is compact with respect to the reference Riemannian metric $G_E$ with respect to $\tau$. Let $d(p,q)$ be the distance function defined by $G_E$. Then for $\bx\in B_{\bx_0}(r_0)$, for $s<s'<s_0$ by \eqref{e-mcf-1} we have:
\bee
d(\bF(\bx,s), \bF(\bx,s'))\le \int_s^{s'}|||\p_s\bF||| d\sigma\le C_2(s'-s)
\eee
for some constant $C_2$ depending only on $\kappa_0, \sup_{M\times[0,s_0)}(|H-\cH|)$. Since $K$ is compact, we conclude that $\bF(\bx,s)$ converges uniformly on $B_{\bx_0}(r_0)$ as $s\to s_0$ to a continuous map  $\bF^{s_0}$ from $B_{\bx_0}(r_0)$ to $N$  with image inside $K$. Let $p= \bF^{s_0}(\bx_0)$. Then there is $\sigma_0>0$ such that $\bF^s(B_{\bx_0},r_0)\Subset D_p(\sigma_0)$ for all $s_1<s<s_0$ for some $s_1>0$, for a possibly smaller $r_0$. Here $D_p(\sigma_0)$ is the geodesic ball with respect to $G_E$ with center at $p$, and radius $\sigma_0$. We may also assume that $D_p(\sigma_0)$ is inside a coordinate neighborhood with coordinates $y^0, y^1,\cdots, y^n$ so that $G_E$ is uniformly equivalent to the Euclidean metric and the components of the metric tensor and its derivatives with respect to $y^a$ are uniformly bounded. Then $\bF$ is of the form:

\bee
\bF =(f^0, f^1,\cdots, f^n).
\eee
In order to prove the lemma, it is sufficient to prove that $\p_s^k\p^\b f^a$ are uniformly bounded, for any $k, \ell\ge 0$ for any multi-index $\b$ for $x^i$ with $|\b|\le \ell$. This is rather standard because of \eqref{e-kappa-bound}, \eqref{e-A-bound} and Lemma \ref{l-evolution-eqs}. We sketch the proof as follows.  Let $\p_i=\frac{\p}{\p x^i}$ be the coordinate frame in $B_{\bx_0}(r_0)$.
Then
\be\label{e-longtime-1}
\left\{
  \begin{array}{ll}
    \bF_*(\p_i)=f^a_i\frac{\p}{\p y^a}; \\
  (H-\cH)\vn= \p_s\bF=f^a_s\frac{\p}{\p y^a}\\
A_{ij}\vn= \lf(f^a_{ij}-\Gamma_{ij}^kf^a_k +f^b_if^c_j\ol\Gamma_{bc}^a\ri)\frac{\p}{\p y^a}.
  \end{array}
\right.
\ee
here $f^a_i, f^a_{ij}$ are partial derivatives with respect to $x^i$, $\ol \Gamma$ is the connection of the target.

By \eqref{e-kappa-bound}, \eqref{e-A-bound} and Lemma  \ref{l-evolution-eqs}(i)--(iii), we can conclude that $g_{ij}(s)$ is uniformly equivalent to $g_{ij}(0)$, and $\Gamma_{ij}^k(s)$ is uniformly bounded. Hence $f^a_i, f^a_s$ are uniformly bounded. By differentiate the equation of $\p_s\Gamma_{ij}^k$ with respect to $x^l$ and integrate, we can conclude that all the partial derivatives of $\Gamma_{ij}^k$ with respect to $x^l$ are uniformly bounded. Hence all the partial derivatives of $A_{ij}$ and $H$ are bounded. Since $|||\cH|||_{k,\Omega_{\tau_1,\tau_2}}<\infty$ for all $k\ge0,$ we conclude that all the partial derivatives of $\cH$ are uniformly bounded using \eqref{e-kappa-bound}. Using \eqref{e-longtime-1}, we may conclude that $ f^a_{ij}$ are uniformly bounded. Take derivatives of the third equation with respect to $x^l$ in \eqref{e-longtime-1}, we conclude that all partial derivatives of $f^a$ with respect to $x^i$ are uniformly bounded. Differentiate the third equation with respect to $x^l$, then we conclude inductively  that $\p^\b f^a$ are uniformly bounded for all multi-index $\b$ for $x^i$. In particularly, $\bF(\bx,s)$ converges in $C^\infty$ in $B_{\bx_0}(r_0)$ to $\bF^{s_0}(\bx)$. Next, we want to investigate $\p_s^k\p^\b f^a$. By Lemma \ref{l-evolution-eqs}(v), (vi) one can conclude that $\p_s^k\p^\b (H-\cH)$, $\p_s^k\p^\b A_{ij}$ are uniformly bounded. For example
\bee
\p_s(H-\cH)=\Delta (H-\cH)-\lf(|A|^2+\ol{\Ric}(\vn,\vn)+\la \on\cH,\vn\ra\ri)(H-\cH).
\eee
Hence $\p_s(H-\cH)$ is uniformly bounded.
Note that
\bee
\p_s\vn=\n(H-\cH)
\eee
and
\bee
\on_{\bF_*(\p_i)}\vn= g^{kl}A_{il}\bF_*(\p_k)=g^{kl}A_{il}f^a_l\frac{\p}{\p y^a}.
\eee
Differentiate with respect to $x^l$, we conclude that $\p_s\p^\b(H-\cH)$ are uniformly bounded. Similarly, $\p_s\p^\b A_{ij}$ are uniformly bounded. By differentiate $\p_s\p^\b(H-\cH)$ with respect to $s$ one can conclude that $\p_s^2\p^\b(H-\cH)$ are uniformly bounded. Here we also use  the fact that $\p_s g_{ij}$ are uniformly bounded. Inductively, one can prove  that $\p_s^k\p^\b (H-\cH)$, $\p_s^k\p^\b A_{ij}$ are uniformly bounded.
Differentiate the second equation in \eqref{e-longtime-1} with respect to $x^l, s$, we conclude inductively that $\p^k_s\p^\b f^a$ are uniformly bounded. Hence the lemma is true.
\end{proof}

\section{Convergence}\label{s-convergence}

We will prove the results on convergence. First we prove Theorem \ref{t-convergence}.

\begin{proof}[Proof of Theorem \ref{t-convergence}] As before, we will denote the immersed surface $\bF(\cdot,s): M\to N$ simply by $M_s$. By  Lemmas \ref{l-grad},  \ref{l-A-est} and the assumptions of the theorem, there exists    $\kappa_1$ and for any $m\ge 0$, there exists $c_m$ such that
\be\label{e-kappa-A-longtime}
\sup_{M\times[0,\infty)}\kappa\le \kappa_1; \sup_{M\times[0,\infty)}|\n^mA|\le c_m.
\ee
At $s=0$, $H\ge \cH-a+\ve_0\ge \ve_0$ and
\be\label{e-bound-1}
|A|^2+\oRic(\vn,\vn)+\la\on\cH,\vn\ra\ge \frac{H^2}n-c^2\ge \frac{\ve_0^2}n-c^2>0.
\ee
We want to prove that this is still true for $s>0$. Since $\tau_-<\tau_1\le \tau(\bF)\le \tau_2<\tau_+$, using \eqref{e-kappa-A-longtime}, Lemma \ref{l-evolution-eqs}(vi) and the assumptions on $\oRm, \cH$, we conclude that  for any $\e>0$, there is $s_0>0$ so that for $0\le s\le s_0$,
$$
|A|^2+\oRic(\vn,\vn)+\la\on\cH,\vn\ra>0.
$$

 By Lemma \ref{l-evolution-eqs}(vi), in $M\times[0,s_0]$,

\bee
\lf(\pps-\Delta\ri)(\cH-H) \le - (|A|^2+\oRic(\vn,\vn)+\la\on\cH,\vn\ra)(\cH-H)
\eee
which is negative at the points where $\cH-H>0$. By Lemma \ref{l-max},
\bee
\sup_{M\times[0,s_0]}(\cH-H)\le \max\{0, \sup_{M_0}(\cH-H)\}\le \max\{0,  a-\e_0\}=a-\e_0.
\eee
because $a>\e_0$. So $H\ge \cH-a+\e_0\ge \e_0$ because $\cH\ge a$. Hence \eqref{e-bound-1} is true on $M\times[0,s_0]$. Because $H-\cH\ge -a+\ve_0$ at $s=s_0$, iterating, we conclude that   \eqref{e-bound-1} is true for all $s$. So by  Lemma \ref{l-evolution-eqs}(vi),
\be\label{e-0-order}
\lf(\pps-\Delta\ri)(H-\cH)^2\le -\ve_1 (H-\cH)^2-2|\n (H-\cH)|^2.
\ee
Using Lemmas \ref{l-exhaustion} and \ref{l-max},   we can conclude that there is a constant $C_0, b_0>0$ such that
\be\label{e-H-cH-decay-1}
\sup_{M_s}|H-\cH|^2\le C_0e^{-b_0 s}.
\ee
By Lemma \ref{l-evolution-eqs}(i), we have that
\be\label{e-metric-1}
\lambda^{-1} g(0)\le g(s)\le \lambda g(0)
\ee
for some $\lambda>0$ for all $s$, where $g(s)$ is the induced metric on $M_s$. Next, we want to show that  for any $k\ge1$, there are   constants $C_k, b_k>0$ such that
\be\label{e-H-cH-decay-2}
\sup_{M_s}|\n^k(H-\cH)|^2\le C_ke^{-b_k s}.
\ee
We will prove this inductively.  To simplify the notation, we write the evolution equation Lemma \ref{l-evolution-eqs}(vi) as
\be\label{e-evolution-H}
(\pps-\Delta)(H-\cH)=\psi (H-\cH).
\ee
where $\psi=-(|A|^2+\oRic(\vn,\vn)+\la\on\cH,\vn\ra).$  By \eqref{e-longtime-1} and the fact that \eqref{e-assumption-1} is true, $|\n^m\psi|$ is uniformly bounded in space time for all $m\ge 0$.
Let us  prove that  the estimate \eqref{e-H-cH-decay-2} is true for $k=1$. The general case can be proved similarly  by induction.

 By the Bochner's formula:
\bee
\begin{split}
(\pps-\Delta)|\n(H-\cH)|^2=&-2|\n^2(H-\cH)|^2-\la \n(\psi (H-\cH)), \n(H-\cH)\ra
\\
&-2\Ric(\n(H-\cH),(\n(H-\cH))\\
\le &-2|\n^2(H-\cH)|^2+K_1\lf(|\n(H-\cH)|^2+(H-\cH)^2\ri)
\end{split}
\eee
in $M\times [0,\infty)$ for some $K_1>0$. Here $\Ric$ is the Ricci tensor of $M_s$ which is uniformly bounded from below because of  \eqref{e-longtime-1} and  \eqref{e-assumption-1}. Combining with \eqref{e-0-order}, there exists $K_2>0$ and $b_2>0$ such that on $M\times[0,\infty)$:
\bee
\begin{split}
(\pps-\Delta)(|\n(H-\cH)|^2 +K_2(H-\cH)^2)\le &
  -2|\n^2(H-\cH)|^2\\
&-b_2(|\n(H-\cH)|^2+K_2(H-\cH)^2).
\end{split}
\eee
From this one can conclude that \eqref{e-H-cH-decay-2} is true for $k=1$. Using \eqref{e-kappa-A-longtime}, \eqref{e-assumption-1}, we conclude that all the derivatives of $\Rm$ on $M_s$ are uniformly bounded in space time. Hence by differentiating \eqref{e-evolution-H}, argue as above we can conclude that   \eqref{e-H-cH-decay-2} is true for all $k$ by induction.

Using Lemma \ref{l-evolution-eqs}(iii), one can conclude that $\p_s\Gamma_{ij}^k$ is of exponential decay and so $\Gamma(s)-\Gamma(0)$ are uniformly bounded in spacetime. Differentiate the equation, we have $\n^k(\Gamma(s)-\Gamma(0))$ are also uniformly bounded in space time. One can proceed as in the proof of Theorem \ref{t-longtime} to conclude that   as $s\to\infty$, $\bF(\cdot,s)$ converges in $C^\infty$ norm in compact sets in $M$ to $\bF_\infty$ so that $\bF_\infty: M\to N$ is an immersed  spacelike hypersurface which is complete in the induced metric because of \eqref{e-metric-1}. By \eqref{e-H-cH-decay-1}, the mean curvature $H_\infty$ of $M_\infty$   is equal to $\cH$.

To prove the last assertion, we follow \cite{Ecker1993}. For any $s_0>0$, by \eqref{e-kappa-A-longtime} and Lemma \ref{l-evolution-eqs}, $\n\phi$ and $\n^2\phi$ are uniformly bounded with respect to $g(s)$ in $M\times[0,s_0]$. Using \eqref{e-0-order}, one obtain
\bee
\begin{split}
\lf(\pps-\Delta\ri)(\phi^2(H-\cH)^2)\le C\phi^2(H-\cH)^2
\end{split}
\eee
on $M\times[0,s_0]$ for some $C>0$. Hence
\bee
\begin{split}
\lf(\pps-\Delta\ri)(e^{-Cs} \phi^2(H-\cH)^2)\le 0.
\end{split}
\eee
By Lemmas \ref{l-exhaustion} and \ref{l-max},  we have
$$
\sup_{M\times[0,s_0]} \phi^2(H-\cH)^2 <\infty.
$$
Combining this with  Lemma \ref{l-evolution-eqs}(vii) and \eqref{e-kappa-A-longtime} and the fact that $\a^{-1}$ is uniformly bounded in $M\times[0,s_0]$, one can conclude that
\bee
|\pps\phi^2(u( s)-u_0)^2|\le C
\eee
in $M\times[0,s_0]$ for some $C>0$. Here $u(s)$ is the height function of $M_s$. This implies that $u(\bx,s)-u_0(\bx)\to 0$ as $\bx\to\infty$ uniformly in $s\in[0,s_0]$.  On the other hand, by Lemma \ref{l-evolution-eqs}(vii), \eqref{e-H-cH-decay-1}, \eqref{e-kappa-A-longtime} and the fact that $\a^{-1}$ is uniformly bounded because the image of $\bF$ is inside $\Omega_{\tau_1,\tau_2}$, we have
\bee
|u_\infty(\bx)-u(\bx,s)|\le C' e^{-b_0s}
\eee
for some $C'>0$ for all $\bx, s$. Combining these two facts, it is easy to see that the last assertion is true.

\end{proof}

The method of proof of the last assertion in the theorem can be used to obtain height estimate of the prescribed mean curvature flow if we have suitable barrier surface. In this case, one may construct spacelike hypersurface with prescribed mean curvature. One example is as follows.  First recall that the {\it past domain of dependence} $D^-(S)$  of an achronal set $S$ is the set consisting of points $p$ in the spacetime so that every inextendible future causal curve from $p$ will meet $S$.  We have the following:

\begin{cor}\label{c-existence}
Let $(N^{n+1},G), \bX_0$, $\tau$, $\rho$, $\cH$ be as in Theorem \ref{t-longtime} so that $M_0 \subset  \Omega_{\tau_1,\tau_2}$ for some $\tau_-<\tau_1<\tau_2<\tau_+$. In addition, assume the following:
\begin{enumerate}
  \item [(i)] $\oRic(w,w)\ge0$ for all time like vector;
  \item [(ii)] $\cH$ is monotone,
     $\cH\ge a>0$ for some constant $a$, $H-\cH\ge0$ on $M_0$, and $\sup_{M_0}\phi(H-\cH)<\infty$ for some smooth function $\phi\ge1$ on $M_0$ with $\sup_{M_0}(|\n\phi|+|\n^2\phi|)<\infty$ and $\phi\to\infty$ as $\bx\to \infty$ in $M_0$;
  \item[(iii)] $\a^{-1}$ is uniformly bounded on $\Omega_{\tau',\tau''}$ for any $\tau_-<\tau'<\tau''<\tau_+$;
      \item[(iv)] there exists an achronal spacelike hypersurface  $M^+$ with mean curvature $H^+<\cH$ on $M^+$, $\tau_+>\sup\tau|_{M^+}\ge\inf \tau|_{M^+}>\tau_2$ such that $M_0\subset D^-(M^+)$.
\end{enumerate}
Then the prescribed mean curvature flow \eqref{e-mcf-1} has long time solution $\bF$ and $\bF$ converges in $C^\infty$ norms on compact sets of $M$ to a complete spacelike hypersurface with prescribed mean curvature $\cH$.

\end{cor}
\begin{proof} Suppose $H\equiv\cH$ at $M_0$, then there is nothing to prove. So we assume that $H-\cH>0$ somewhere initially.  By Theorems \ref{t-longtime} and \ref{t-convergence} it is sufficient to obtain height estimates for solution of \eqref{e-mcf-1}. Let $\bF$ be a solution on $M\times[0,s_0]$ so that $\sup_{M\times[0,s_0]}(|\n^kA|+\kappa)<\infty.$ As in the proof of the last assertion of Theorem \ref{t-convergence}, if we let $u=\tau(\bF)$ be the height function of $M_s$, then $u(\bx,s)-u_0(\bx)\to 0$ as $\bx\to\infty$ uniformly on $s$, where $u_0=\tau|_{M_0}$. Hence there is $\e>0$ and a compact set $K$, such that
$u\le \tau_{M^+}-\e$ for all $\bx\notin K$ and for all $0\le s\le s_0$. By Lemma \ref{l-evolution-eqs}(vi) and assumptions (i) and the fact that $\cH$ is monotone, we have
\bee
\lf(\pps-\Delta\ri)(H-\cH)=\psi (H-\cH)
\eee
with $\psi\le 0$. Hence Lemma \ref{l-max} and the strong maximum principle \cite[Theorem 2.7]{Lieberman}, $H-\cH>0$ if $s>0$. By \eqref{e-mcf-1}, $\p_s\bF$ is future pointing time like vector for $s>0$. Hence for any $\bx\in M$,  $\bF(\bx,s)$, $0\le s\le s_0$ is a future nonspacelike curve. Since $\tau$ is nondecreasing on such a curve, we have $\tau(\bF(\bx,s))\ge \tau_1$ for all $\bx\in M$, $s\in [0,s_0]$. To obtain an upper bound, we proceed as in \cite{EckerHusiken1991}.  Note that $\bF(\bx,s)$, $0\le s\le s_0$   is part of an inextendible future nonspacelike curve. Suppose $\bF(\bx,s)$, $0\le s\le s_0$ does not intersect $M^+$, then $\tau(\bF(\bx,s))\le \sup\tau|_{M^+}$ because $\tau$ is  nondecreasing on a future causal curve. Hence we may assume that $M_s\cap M^+\neq \emptyset$ for some $0<s\le s_0$. Since $u\le \tau_{M^+}-\e$ for all $\bx\notin K$ and for all $0\le s\le s_0$, we conclude that there exists $0<s_1\le s_0$ and $\bx_1\in K$ so that $\bF(\bx_1,s_1)\in M^+$ and for all $0\le s< s_1$, $\bF(\bx,s)$ is in the past of $M^+$. So $M_{s_1}$ is tangent to $M^+$ at $\bF(x_1,s_1)$. Moreover, at $(\bx_1,s_1)$, $H-\cH\ge0$. Hence $H>H^+$ at $\bF(\bx_1,s_1)$ this contradicts the strong maximum principle. Hence we must have $\tau(\bF)\le \le \sup\tau|_{M^+}$. One then apply Theorems \ref{t-longtime}, \ref{t-convergence} to conclude the result.

\end{proof}

Next, we want to prove a convergence result without assuming $\cH>\delta>0$.

\begin{proof}[Proof of Theorem \ref{t-convergence-1}]
Since $H-\cH\ge0$ in $M_0$,  $\oRic(w,w)\ge 0$, $\cH$ is monotone, as in the proof of Theorem \ref{t-convergence}, one can conclude that $H-\cH\ge0$ for all $s$.  Since $\cH\ge0$, we have $H\ge H-\cH\ge0$. By Lemma \ref{l-evolution-eqs}(vi), we have:
\bee
\lf(\pps-\Delta\ri)(H-\cH)^2\le -(H-\cH)^4.
\eee
Comparing the function $s^{-1}$ and $(H-\cH)^2$, using Lemmas \ref{l-exhaustion} and \ref{l-max}, we have
\be\label{e-decay-H-cH-3}
(H-\cH)^2(\bx,s)\le s^{-1}
\ee
for all $\bx\in M$.
  Now we claim that for any compact set $K\subset M$, $\rho(\bF(\bx,s))\le \rho_0$ for some $\rho_0$ for all $(\bx,s)\in K\times[0,\infty)$. Let $u(\bx,s)=\tau(\bF(\bx,s))$ be the height function, then in $K\times[0,\infty)$, by Lemma \ref{l-evolution-eqs}(vii), we have
$$
\pps u=\a^{-1}\kappa(H-\cH)\ge C_1\kappa(H-\cH)
$$
for some $C_1>0$ because $\a$ is bounded from above, and
\bee
\pps\rho(\bF(\bx,s))=\la \on\rho, \p_s\bF\ra(\bx,s)\le C_2\kappa(H-\cH)
\eee
for some $C_2>0$ by \eqref{e-mcf-1},  the fact  that $|||\on\rho|||_{\Omega_{\tau_1,\tau_2}}$ is finite, and that  $H-\cH\ge0$. Hence
\bee
\pps u-\delta \rho\ge0
\eee
for some $\delta>0$ on $K\times[0,\infty)$. Hence
\bee
u(\bx,s)-\delta\rho(\bF(\bx,s))\ge u(\bx,0)-\delta\rho(\bF(\bx,0)).
\eee
Since $\tau_1\le u\le \tau_2$, it is easy to see the claim is true.
From this one can conclude that $\bF(K\times[0,\infty))\subset W$ for some compact set $W$ in $N$. As in the proof of Theorem \ref{t-convergence}, since the tilt factor $\kappa$ and all the covariant derivatives of the second fundamental form of $\bF(\cdot,s)$ are uniformly bounded in space and time, using \eqref{e-decay-H-cH-3}  one may proceed as in \cite[Theorem 3.8]{Bartnik1988} and \cite{Ecker1993} to conclude that the theorem is true. The idea is to cover $W$ with finitely many suitable simply connected coordinate neighborhoods so that one can apply Lemma \ref{l-convergence}.  Note that the limit is strictly spacelike because $\kappa$ is uniformly bounded.

\end{proof}

\begin{rem}\label{r-dual}
It is easy to see that Theorems \ref{t-convergence}, \ref{t-convergence-1} are true under dual conditions. For example, in Theorem \ref{t-convergence-1} if we assume $\cH\le 0$, $H-\cH\le 0$ in $M_0$ instead, then the result is still true. We can appeal to the theorem by changing the time orientation and replaced $\cH$ by $-\cH$.

\end{rem}

\section{Examples}\label{s-examples}

\subsection{Application on Minkowski spacetime}\label{ss-Mink}

We will apply the previous results to study some prescribed mean curvature flow in Minkowski spacetime.
 Consider the Minkowski spacetime $\R^{1,n}$ with metric
 $$g_{\mathrm{Mink}}=-dt^2+\sum_{i=1}^n (dx^i)^2.$$
 We consider $\ppt$ as future pointing. Consider the family of spacelike surfaces with constant mean curvature $1/\tau>0$  given by
$$
S_\tau=\{(t,\bx)|\ t>0, t^2-r^2=\tau^2\}
$$
where   $\bx=(x^1,\cdots,x^n)$ and $r=|\bx|$. Then $S_\tau$ will foliate $N=I^+(0)$, the future  of the origin. We will use $\tau$ as time function, see Lemma \ref{l-time function-1} below.

\bee
\tau_-=\inf_N\tau=0,\ \  \tau_+=\sup_N\tau=+\infty.
\eee

As an application of previous results, we want to prove the following:
\begin{prop}\label{t-Mink}
\begin{enumerate}
  \item [(i)] A small  perturbation of $S_\tau$ in a compact set will be deformed under a prescribed mean curvature flow to a constant mean curvature surface with mean curvature $1/\tau$. Namely, if $M$ is a spacelike surface which coincide with $S_\tau$ outside a compact set of $S_\tau$ so that the mean curvature $H$ of $M$ satisfies $ H\ge \ve_0$ for some $\ve_0>0$. Then the prescribed mean curvature flow with $\cH=\frac1\tau$ has long time solution and will converge  to a complete spacelike hypersurface $M_\infty$ of constant mean curvature $1/\tau$. Moreover, the height function $u_\infty$ will tend to $\tau$ at infinity.

  \item [(ii)] Let $\cH$ be a smooth function in $N$ which is monotone and satisfies the condition in \eqref{e-assumption-1}. Moreover, suppose $|\cH-\frac1 {\tau_0}|<\frac1{2\tau_0}$ on $S_{\tau_0}$ and $\sup_{S_{\tau_0}}\phi|H-\cH|<\infty$, where $\phi\ge 1$ is an exhaustion function with bounded gradient and Hessian on $S_{\tau_0}$. Then the prescribed mean curvature flow with prescribed mean curvature $\cH$ has a long time solution starting from $S_{\tau_0}$ and will converge to a complete spacelike hypersurface $M_\infty$ of prescribed mean curvature $\cH$. Moreover, the height function $u_\infty$ will tend to $\tau$ at infinity.
\end{enumerate}
\end{prop}

We will give some examples of $\cH$ in (ii) after we prove the theorem. First, we need the following lemma.
\begin{lma}\label{l-time function-1} We have the following:
\begin{enumerate}
  \item [(i)] $\on \tau=-\frac1\tau\lf(t\ppt+r\frac{\p}{\p r}\ri)$, $\la\on\tau,\ppt\ra=\frac t\tau$. Hence $\on \tau$ is past directed timelike. The lapse function $\a$ of $\tau$ is 1.
      \item[(ii)] Let $G_E$ be the reference Riemannian metric with respect to $\tau$ and let $\bT=-\on\tau$. Then  $|||\on \bT|||_{\Omega_{\tau_1,\infty}}\le C\tau_1^{-1} $ and $|||\on^2 \bT|||_{\Omega_{\tau_1,\infty}}\le C\tau_1^{-2} $ for some constant $C$ for all $\tau_1>0$.
          \item[(iii)] Let $\rho=\log (r+2)>0$. Then $|||\on\rho|||_{\Omega_{\tau_1,\infty}}<\infty$ for all $\tau_1>0$. Moreover, for any $0<\tau_1<\tau_2<\infty$ and any $\rho_0>0$, then set $\ol\Omega_{\tau_1,\tau_2}\cap\{\rho\le \rho_0\}$ is compact.
\end{enumerate}

\end{lma}
\begin{proof} (i) follows by direct computation.

(ii) Since $\on \tau=-\frac1\tau\lf(t\ppt+x^i\frac{\p}{\p x^i}\ri)$, at $(\tau_1,0)$ the first derivatives and the second derivatives of the coefficients of $\ppt, \frac{\p}{\p x^i}$ are uniformly bounded by $C\tau_1^{-1}$ and $C\tau_1^{-2}$ for some constant $C$. Since $\on \tau=-\ppt$, one can conclude that $|||\on \bT||, |||\on^2\bT|||$ are bounded by $C\tau_1^{-1}$ and $C\tau_1^{-2}$  at $(\tau_1,0)$,  for some constant $C$. Since for any point
$p=(t, x^1,\cdot,x^n)\in S_{\tau_1}$, there is a restricted  Lorentz transformation  $\Phi$ which is an isometry of $N$, and maps $(\tau_1, 0)$ to $p$.  Moreover, $\tau\circ\Phi(q)=\tau(q)$ for all $q$. From this we can conclude that (ii)  is true.

(iii) Let $G_E'$ be the reference Riemannian metric with respect to$\ppt$. Then $|||\on \rho|||_{G'_E}\le \frac1{r+2}$. By (i) and Lemma \ref{l-tilt-equivalent}, we have
\bee
|||\on\rho|||\le    \frac {2t}\tau \cdot \frac1{r+2}=\frac{2(\tau^2+r^2)^\frac12} {\tau(r+2)}\le 2(1+\frac1{\tau_1}).
\eee
It is easy to see that (iii) is true.

\end{proof}

\begin{proof}[Proof of Proposition \ref{t-Mink}] (i) By Theorems \ref{t-shorttime Lorentz},   \ref{t-longtime}, \ref{t-convergence}, it is sufficient to show that if $\bF$ is a solution of the prescribed mean curvature flow \eqref{e-mcf-1} with $\cH= 1/\tau_0$ starting from $M$, with $0\le s\le s_0$,  so that $\kappa$ and $|\n^mA|$ are uniformly bounded on $M\times[0,s_0]$, then $0<\tau_1\le \tau(\bF)\le \tau_2<\infty$ for some $\tau_1, \tau_2$ independent of $s$. As in the proof of the last part of Theorem \ref{t-convergence},
the height function $u$ of $M_s$ will tends to $\tau_0$ at infinity uniformly for $s\in[0,s_0]$. Choose $\tau_2>\tau_0$ and $\tau_2>\sup_{M_0}u_0$, where $u_0$ is the height function of $M_0=M$. Since the mean curvature of $S_{\tau_2}=1/\tau_2<1/\tau_0$, by the maximum principle as in \cite[p.605]{EckerHusiken1991}, we can conclude that $\tau(\bF)<\tau_2$. Note that $\tau_2$ does not depend on $s_0$. Similarly, one can prove that $\tau(\bF)\ge \tau_1>0$ where $\tau_1$ does not depend on $s_0$. From this the result follows.

 (ii)  By the assumption, we have $\cH>\frac 1{2\tau_0}$ and $|H-\cH|\le  \frac 1{2\tau_0}-\ve_0$ for some $\ve_0>0$. Hence the proof  is similar to the proof of (i).

\end{proof}

{\it Example}: In the Minkowski space, let $f(t,\bx)=e^{-4t+(r^2+1)^\frac12}$ where $r=|\bx|$. We want to check that $\cH(t,\bx)=2-f(\bx,t)$ satisfies the assumptions in  Proposition \ref{t-Mink}(ii) with $\tau_0=\frac12$. We use $(r^2+1)^\frac12$ because $r$ is not a smooth function. Let us first check that $\sup_{S_\frac12}|\cH-2|<1$ on $S_\frac12$
On $S_\frac12$, $t^2= r^2+\frac14$ and $t\ge r, t\ge \frac12$ and so $f\le e^{-\frac12}<1$. Hence $\sup_{S_\frac12}|\cH-2|<1$. Next we want to check that $\cH$ is monotone. Let $w$ be a future directed timelike vector so that $w=a_0\ppt+\sum_{i=1}^n a_i\frac{\p}{\p x^i}$. Then $a_0\ge |\mathbf{a}|$ where $\mathbf{a}=(a_1,\dots, a_n)$. Now $\on \cH=f\lf(-4\ppt- (1+r^2)^{-\frac12}\sum_{i=1}^n x^i\frac{\p}{\p x^i}\ri)$. Hence
\bee
\la \on\cH,w\ra=f\lf(4a_0-(1+r^2)^{-\frac12}\sum_{i=1}^na_i x^i\ri)\ge f\lf(4a_0-r(1+r^2)^{-\frac12}|\mathbf{a}|\ri)\ge 0.
\eee
Next, we want to show that $|||\on^k\cH|||_{\Omega_{\tau_1,\tau_2}}<\infty$ for any $0<\tau_1<\tau_2<\infty$. It is easy to see that $\cH$ is uniformly bounded in $N$.
Let $G_e$ be the Euclidean metric $dt^2+\sum_{i^1}^n (dx^i)^2$.  For any $k\ge 1$
$$
|\on^k\cH|_{G_e}\le Cf
$$
for some constant $C$ depending only on $k, n$. By Lemmas \ref{l-tilt-equivalent} , \ref{l-time function-1} , we have
\bee
|||\on^k\cH|||\le Cf\cdot\frac{t^k}{\tau^k}\le C'\tau_1^{-k}
\eee
for some constant $C'$. Here we have used the fact that $t\ge r$ in $\Omega_{\tau_1,\tau_2}$. Finally, we claim that $  d(p)|H-\cH|(p)\le C''$ on $S_\frac12$ for some constant $C''$ where $d(p)$ is the intrinsic distance of $S_\frac12$ form a fixed point. At a point $p=(t,\bx)$ in $S_\frac12$, one can check that $d(p)\sim \log r$, where $r=|\bx|$. From this it is easy to see the claim is true. Since the distance function on $S_\frac12$ has bounded gradient and Hessian away from $d=0$, hence all the conditions in Proposition \ref{t-Mink} are satisfied and we  obtain a complete spacelike hypersurface in $N$ with prescribed mean curvature given by $\cH$.

\subsection{Computation near the future null infinity of the Schwarzschild spacetime}\label{ss-Schwarzschild}

 We want to do some computations which might be related to the study of prescribed mean curvature flow in a asymptotically Schwarzschild spacetime introduced in \cite{AnderssonIriondo1999}. Let us first recall the future null infinity of the Schwarzschild spacetime.  The standard Schwarzschild metric in $r>2m>0$ is
\be\label{e-Sch-metric}
g_\mathrm{Sch}=-\lf(1-\frac{2m}r\ri)dt^2+\lf(1-\frac{2m}r\ri)^{-1}dr^2+r^2 \sigma.
\ee
defined on $\{r>2m\}$, $-\infty<t<\infty$ where $r=\sum_{i=1}^3(x^i)^2$ with $(x^1,x^2,x^3)\in \R^3$ and $\sigma$ is the standard metric of the unit sphere $\mathbb{S}^2$. Consider the retarded null coordinate

\be\label{e-uv}
     v=t-r_*
\ee
where $r_*=r+2m \log\lf(\frac r{2m}-1\ri)$ in the region $r>2m$. Let $x=r^{-1}$,

\be\label{e-metric-null}
\begin{split}
g=g_{\mathrm{Sch}}=&-(1-2mx)d {   v}^2+2x^{-2}  d {  v}dx+x^{-2} \sigma\\
=:&\,x^{-2}\wt g,
\end{split}
\ee
with $ 0<x<\frac1{2m}$, $-\infty<v<\infty$.
Here the unphysical metric $\wt g$ is
the product metric:
\be\label{e-unphysical-1}
\wt g=(\sigma_{AB})\oplus\left(
      \begin{array}{cc}
        0 & 1 \\
        1 &-x^2(1-2mx) \\
      \end{array}
    \right).
\ee
Here $(\sigma_{AB})$ is the standard metric for $\mS^2$ in local coordinates $y^1, y^2$. $\wt g$ can be extended   smoothly   is Lorentz up to $x=0$. Then future null infinity $\mathcal{I}^+$ can be identified as $x=0$ in the compactification.   In the following, let $h=1-\frac{2m}r=1-2mx$.
Given a smooth function $f(\by)$ on $\mathbb{S}^2$. Consider the cut   $\mathcal{C}$ given by $(\by,   f(\by)), \by\in \mathbb{S}^2$ in $\cI^+$.
For $\tau>0$, let

\be\label{e-foliation-1}
 P(  \by, x, \tau)=f(\by)+x\phi(\tau,\by)+\frac1{2!}x^2\psi(\tau,\by),
\ee
where $\phi=P_x, \psi=P_{xx}$ at $s=0$ are smooth functions in $\tau, \by$, given by

\be\label{e-phi-psi}
\left\{
  \begin{array}{ll}
   \phi=-\frac12\lf(\tau^2+|\wn f|^2_{\mS^2}\ri); \\
   \psi=\frac12\lf(\tau^2\wt\Delta f+\la \wn|\wn f|^2_{\mS^2},\wn f\ra_{\mS^2}\ri).
  \end{array}
\right.
 \ee
 Here $\wn, \wt\Delta$ are the covariant derivative and the Laplacian of the standard $\mS^2$.
 By  \cite[Theorem 3.1]{LST}, so that if  $\Sigma_\tau$ is the surface given by $(\by, x)\to (\by, x, -P)$ in the $\by, x, v$ coordinates, then $\Sigma_\tau$ is spacelike near $x=0$.
Let $0<\tau_1<\tau_2<\infty$ be fixed. Let
$$
M=\{\by\in \mathbb{S}^2, x\in (0,\frac1{2m}),\tau\in (\tau_1,\tau_2)\}=\mS^2\times(0,s_0)\times(\tau_1,\tau_2).
$$
Consider the map $\Phi$ from $M$ to the Schwarzschild spacetime in $\by, x, v$ coordinates defined by:
\be\label{e-Phi}
\Phi(\by, x, \tau)=(\by, x, v(\by,x, \tau))
\ee
with  $v(\by,x,\tau)=-P(\by, x,\tau)$. Then there is $s_0>0$ such that $\Phi$ is diffeomoprhic to its image. Hence $\tau$ can be considered a function in the image. We also have $\on \tau$ is past directed time like vector field if $x_0>0$ is small enough. See \cite{Tam2024}. Let $\bT=-\a\on\tau$, where $\a$ is the lapse function of $\tau$.
We have the following facts which are related to the assumptions in the results in previous sections.

\begin{prop}\label{p-Sch-2} In the above setting, let $N=\Phi(M)$ and $g_\mathrm{Sch}$ be the Schwarzschild metric and let $\Theta$ be the reference Riemannian metric with respect to $\tau$. Denote the norm with respect to $\Theta$ by $|||\cdot|||$. Then there exists $x_0>0$ such that the following are true.
\begin{enumerate}
  \item [(i)] $|||\on \bT|||_{1,N}$ is finite.
  \item [(ii)] $\a, \a^{-1}$ and $|||\on\a|||$ are uniformly bounded.
  \item[(iii)] $|||\on\log r|||$ is uniformly bounded.

  \item[(iv)] The second fundamental form $A$ and its covariant derivatives of the level set $\tau$=constant are uniformly bounded.
      \item [(v)] $|||\oRm|||_{k,N}$ is finite for all $k\ge 0$.
      \item[(vi)] The mean curvature of the level set $\tau$=constant satisfies:
      $H=\frac1\tau+O(x^2)$ and that $d(p)|H(p)-\frac1\tau|\to 0$ as $p\to\infty$ on $\tau$=constant, where $d(p)$ is the intrinsic distance from a fixed point on the level surface.
\end{enumerate}
  \end{prop}
  \begin{proof}

(i) follows from \cite[Lemmas 3.7, 3.8]{Tam2024}.

(ii) follows from \cite[Lemma 3.2]{Tam2024}.

(iii) On the other hand, let $\wt\bT=h^{-\frac12}\ppt$ in the standard coordinates $t, x^i$ for the exterior of the Schwarzschild spacetime. Let $G_E'$ be the reference Riemannian metric with respect to $t$.  One can compute by (ii):

\be\label{e-T-T}
\la \bT,\wt\bT\ra_G=\a h^{-1}\la \on\tau,\ppt\ra_G=-\a h^{-1}P_\tau^{-1}=O(x^{-1})
\ee
near $x=0$. From this it is easy to see that (iii) is true.

(iv) The bound of $A$ follows from the \cite[Lemma 3.4]{Tam2024}. From the proof of this fact, one can also show that   $\n^k A$ is also uniformly bounded for all $k\ge 1$.

(v)
  To prove $|||\oRm|||$ is bounded, we need to prove that $\oRm$ is bounded when evaluated at an orthonormal basis $w_1, w_2, w_3, w_4$ where $w_4=\bT$ and $w_1,w_2, w_3$ are tangential to the level surface of $\tau$. Note that one cannot just appeal to Lemma \ref{l-tilt-equivalent}  to conclude that $|||\oRm||| $ is bounded because $|||\oRm|||_{G_E'}$ decays like $x^3$, $\oRm$ is a four tensor, and we only have \eqref{e-T-T}.   Let us choose $w_i$, $1\le i\le 3$ as follows. Let $y^1, y^2$ be local coordinates of $\mS^2$, and let $\Phi$ be the diffeomorphism in \eqref{e-Phi}:
\be\label{e-frame-1}
\left\{
  \begin{array}{ll}
e_A=:\Phi_*(\frac{\p}{\p y^A})=-P_{A} \p_v+\p_{A}, A=1, 2;\\
e_3=:\Phi_*(\frac{\p}{\p x })=-P_x\p_v+ \p_x;\\
  e_4=:\Phi_*(\frac{\p}{\p\tau})=-P_\tau \p_v.
  \end{array}
\right.
\ee
Then $e_1, e_2, e_3$ are tangential to the level set of $\tau$. We have
\be\label{e-tangential norm}
\left\{
  \begin{array}{ll}
    \la e_A, \wt\bT\ra = -P_Ah^\frac12&\hbox{$A=1, 2$;} \\
    \la e_3, \wt\bT\ra = -P_xh^\frac12.  \\
  \end{array}
\right.
\ee
Since $\la e_i,e_i\ra=P_i^2+r^2=P_i^2+x^{-2}$, for $i=1,2,3$ we have
$$
|||e_i|||_{G_E'}^2= \la e_i,e_i\ra+2 (\la e_i, \wt\bT\ra)^2 \le Cx^{-2}.
$$
We may obtain orthonormal basis $w_i$ from $e_i$ so that $w_i=xc_{ij}e_j$, with $c_{ij}$ being bounded, see \cite[Lemma 3.1]{Tam2024}.
$|||w_i|||_{G_E'}\le C$ for some constant $C$. From \eqref{e-T-T}, we also have $|||w_4|||_{G_E'}=|||\bT|||_{G_E}\le Cx^{-1}$. To summarise:
\be\label{e-norm}
\left\{
  \begin{array}{ll}
    |||w_i|||_{G_E'}\le C, & \hbox{for $1\le i\le 3$;} \\
   |||w_4|||_{G_E'}\le Cx^{-1},
  \end{array}
\right.
\ee
for some constant $C>0$, provided $s_0$ is small enough.
Now we are ready to estimate $|||\oRm|||$. Since $\oR(w_4,w_4,\cdot,\cdot)=0$, by \eqref{e-norm}, we have
\bee
|\oR(w_a,w_b,w_c,w_d)|\le |||\oRm|||_{G_E'}|||w_a|||\, |||w_b|||\, |||w_c|||\, |||w_c|||\le Cx.
\eee
because $|||\oR|||_{G_E'}\le Cx^3$. Therefore, $|||\oRm|||\le Cx$. Similarly, one can conclude that $|||\on^k\oRm||| \le Cx.$

(vi) By \cite[Lemma 2.1]{LST}, The fact that $|H-\frac1\tau|=O(x^2)=O(r^{-2})$ follows from \cite[Theorem 2.1]{LST} and the definition of $P$. Since $\tau$=constant can be written as a spacelike graph over $\bx$ in the standard coordinates $t,x^i$, as before  one can see the last assertion is true.
\end{proof}

\appendix
\section{H\"older spaces }
We begin by explicitly defining the local H\"older norms used in the definition  parabolic H\"older spaces.  Let $\Omega$ be an open set in $\R^m$. Let $k\ge0$ be an integer and $0<\sigma<1$,  for $s_0>0$, the $C^{2k+\sigma, k+\frac\sigma2}$ norm on $ \Omega_{s_0}=\Omega\times[0,s_0]$  for functions $f(\bx,s)$  is defined as:
 \begin{equation*}
 \begin{split}
 ||f||_{2k+\sigma,k+\frac\sigma2;\Omega_{s_0}}=&\sum_{|\a|+2r=0}^{2k}\sup_{\Omega}\lf| \p^r_s\p^\a f \ri|\\
 &+\sum_{|\a|+2r=2k}\sup_{(\bx,s)\neq (\bx',s')\in  \Omega_{s_0}} \frac{\lf| \p^r_s\p^\a f(\bx,s)-\p^r_s\p^\a f(\bx',s')\ri|}{|\bx-\bx'|^\sigma+|s-s'|^\frac\sigma2}
  \end{split}
\end{equation*}
 where
 $$\p^r_s\p^\a u=\frac{\p^{r+|\a|}u}{\p s^r\p x^\a} $$
 and $\a$ is a multi-index. If there is no confusion on the domain  of definition of $f$, we will simply write the norm as  $||f||_{2k+\sigma,k+\frac\sigma2}$.

\section{A convergence lemma}

Let $U$ be a coordinate neighborhood of a Lorentz manifold with time function $\tau$, $\on \tau$ is past directed, so that $U$ is given by
$$
U =(-2,2)\times B(2)=\{(x^0,x^i,\dots,x^n)| |x^a|<2, 0\le a\le n\}
$$
where  $\tau=x^0$. Let
$$
V=(-1,1)\times B(1)=\{(x^0,x^i,\dots,x^n)| |x^a|<1, 0\le a\le n\}
$$
Moreover the metric is of the form $G=-\a^2d\tau^2+g_{ij}dx^idx^j$, which is zero shift. Assume the lapse function $\a$ is uniformly bounded from above and below. We also assume that $g_{ij}$ is uniformly equivalent to the Euclidean metric $\delta_{ij}$.
\begin{lma}\label{l-convergence} Let $M_k$ be a sequence of spacelike surfaces so that
\begin{enumerate}
  \item [(i)] $M_k\cap U$ is connected, $\p M_k\cap U=\emptyset$;
  \item [(ii)] the mean curvature $H_k$ of $M_k$  together with all its derivatives are bounded; and
  \item [(iii)] the tilt factor $\kappa_k$ of $M_k$ with respect to $\tau$ are uniformly bounded.
\end{enumerate}
If there is a limit point $p$ of $M_k$ inside $V$, then there is a subsequence of $M_k$, still denoted by $M_k$  and a spacelike hypersurface $M_\infty$ inside $V$ so that $d_h(M_k,M_\infty)\to 0$ in every compact subsets of $V$. Moreover, the mean curvature $H_\infty$ is such that $\lim_{k\to\infty}H_k=H_\infty$.
\end{lma}
\begin{proof} First we may assume that $H_k\to H_\infty$ in $C^\infty$ norm in any compact set of $U$. Since $U$ is simply connected, $M_k$ is achronal \cite[p.427]{ONeill}. Hence each $M_k$ is a graph. Namely, for each $k$, there is a connected open set $W_k\subset U$ and a smooth function $u_k$ such that
\bee
M_k=\{(u_k(\bx), \bx)|\ \bx\in W_k\}
\eee
where $\bx=(x^1,\cdots, x^n)$. Since $M_k$ is spacelike, there is $C_1>0$ such that
$|Du_k|\le C_1$ for all $k$, where $Du_k$ is the gradient with respect to the Euclidean metric. Hence there is $\delta>0$ independent of $\bx$ such that if $(u_k(\bx), \bx)\in M_k\cap V$, with $|u_k(\bx)|\le 1-\eta$ for some $\eta>0$.  Then $B_{\bx}(\delta) \subset B_{\bx}(2\delta)\subset W_k$ and $|u_k(\by)|<1$ for $\by\in B_{\bx}(\delta).$ Here $B_{\bx}(\delta)=\{\by|\ |y^i-x^i|<\delta, 1\le i\le n\}$.
Since the tilt factors of $M_k$ are uniformly bounded  and $H_k\to H_\infty$, the equation satisfies by $u_k$ is uniformly elliptic in $B_{\bx}(\delta)$. Hence if $p$ is a limit point of $M_k$ inside $V$, passing to a subsequence, $u_k$ will converges in $C^\infty$ norm on compact sets of $B_{\bx}(\delta)$ where $\bx$ is the projection of $p$ to $B(2)$. From this it is easy to see the lemma is true.
\end{proof}

\end{document}